\documentclass[11pt, twoside, reqno]{amsart}

\usepackage[utf8]{inputenc}
\usepackage{lmodern}

\usepackage{hyperref} 
\usepackage{xcolor}    
\usepackage{esint}
\usepackage{accents}

\definecolor{darkcyan}{cmyk}{1, 0, 0, 0.6}
\hypersetup{
    colorlinks=true,
    linkcolor=darkcyan,  
    urlcolor=blue,    
    citecolor=darkcyan 
}

\usepackage{amsaddr}

\usepackage{amssymb}
    \usepackage{amsmath}
\usepackage{enumerate}
\usepackage{tikz}
\usepackage{comment}

\usepackage{mathrsfs}
\usepackage{stmaryrd}
\usepackage{enumitem}
\usepackage{dirtytalk}
\usepackage{tikz-cd} 
\newtheorem{thm}{Theorem}[section]
\newtheorem{prop}[thm]{Proposition}
\newtheorem{coro}[thm]{Corollary}
\newtheorem{lemma}[thm]{Lemma}
\theoremstyle{definition}

\newtheorem{defi}[thm]{Definition}
\theoremstyle{remark}
\newtheorem{remark}[thm]{Remark}
\numberwithin{equation}{section}
\usepackage[margin=1.2in]{geometry}
\newcommand*\dif{\mathop{}\!\mathrm{d}}

\newcommand{\eps}{\epsilon}

\newcommand{\ubar}[1]{\underaccent{\bar}{#1}}

\newcommand{\R}{\mathbb R}

\newcommand{\mc}[1]{\mathcal{#1}}

\newcommand{\mr}[1]{\mathrm{#1}}
\newcommand{\bs}[1]{\boldsymbol{#1}}
\newcommand{\br}[1]{\bs{\mathrm{#1}}}
\newcommand{\ms}[1]{\mathsf{#1}}
\newcommand{\msc}[1]{\mathscr{#1}}

\newcommand{\wt}[1]{\widetilde{#1}}
\DeclareMathOperator*{\esssup}{ess\,sup}
\DeclareMathOperator*{\essinf}{ess\,inf}

\usepackage{orcidlink}

\title[Time regularity and decay for $2\times 2$ conservation laws]{Strong time regularity and decay of $\mathbf L^\infty$ solutions to $2\times 2$ systems of conservation laws}
\author{Luca Talamini}
\address{Mathematics Area, SISSA, Trieste}
\email{ltalamin@sissa.it}
\begin{document}

\maketitle

\begin{abstract}
    We consider $\mathbf L^\infty$ solutions to $2\times 2$ systems of conservation laws. For genuinely nonlinear systems we prove that finite entropy solutions (in particular entropy solutions, if a uniformly convex entropy exists) belong to $C^0(\mathbb R^+; \mathbf L^1_{loc}(\mathbb R))$. 
Our second result establishes a dispersive-type decay estimate for vanishing viscosity solutions.
Both results are unified by the use of a kinetic formulation.
\end{abstract}

\section{Introduction}
In this paper we study $\mathbf L^\infty$ entropy solutions $\bs u: \mathbb R^+ \times \mathbb R \to  \mathbb R^2$ to $2\times 2$ systems of conservation laws in one space dimension
\begin{equation}\label{eq:systemi}
\partial_t\,\bs  u(t,x) + \partial_x \, f(\bs  u(t,x)) = 0, \qquad \text{in $\mathscr D^\prime_{t,x}$} \qquad \bs u \in \mc U \subset \mathbb R^2, \quad f: \mc U \to \mathbb R^2.
\end{equation}
Entropy solutions are weak solutions to \eqref{eq:systemi} that in addition satisfy the entropy inequality 
 \begin{equation}
     \partial_t \eta(\bs u) + \partial_x q(\bs u) \leq 0 \qquad \text{in $\mathscr D^\prime_{t,x}$}
 \end{equation}
 for every entropy-entropy flux pair $(\eta(\bs u), q(\bs u)) \in \mathbb R \times \mathbb R$ such that 
 $$
 \nabla q(\bs u) = \nabla \eta(\bs u) D f(\bs u), \qquad \text{$\eta$ convex.}
 $$

It is by now a classical result that, under \emph{a priori} $\mathbf L^\infty$ and energy bounds on a sequence of approximate solutions $\bs u^\varepsilon$, the method of compensated compactness \cite{DiP83a, Tar79} allows to prove the strong compactness of the family $\{\bs u^\varepsilon\}_{\varepsilon}$, under standard nonlinearity assumptions on the flux $f$, known as \textit{genuine nonlinearity} (see Definition \ref{defi:GNL}), thus proving existence of solutions in some important cases (see e.g. \cite{LPT96}). A general existence result is lacking, and the assumptions allowing to apply the method must be checked each time. For an account on this topic we refer to \cite[Chapter 9]{Ser00}, \cite{Daf16}.

A long standing problem is the one of understanding the regularity and other qualitative properties of solutions constructed with the compensated compactness method (see e.g. the open questions in \cite[Section 1.13]{Per02}). It is known that merely continuous weak solutions to \eqref{eq:systemi} can be quite wild, in particular they are non-unique \cite{CVY24}; however entropic solutions may still exhibit better regularity or stability properties. It is conjectured that uniqueness holds at least in some intermediate spaces  (see \cite{ABB23, BMG25, ABM25} for some recent results), in between small $BV$ (where uniqueness is known \cite{BB05, BDL23, CKV22}) and small $\mathbf L^\infty$ (where existence of solutions with good decay/regularity properties is known \cite{GL70, BCM10, Gla24}, but nothing is known about uniqueness).  A few special exceptions for which some results about the structure and regularity of large solutions are available are systems of Temple class \cite{AC05} (systems for which rarefaction and shock curves coincide) and $2\times 2$ systems with a specific structure on the eigenvalues (i.e. the eigenvalue $\lambda_i$ depends only on the $i$-th Riemann invariant), for which a notable example is the system of isentropic gas dynamics with $\gamma = 3$ (\cite{LPT94b}, \cite{Gol23}, \cite{Tal24}, \cite{Vas99}). 

For general genuinely nonlinear systems, in \cite{AMT25} it is shown that finite entropy solutions to \eqref{eq:systemi} satisfy a kinetic equation (based on a family of discontinuous entropies constructed in \cite{PT00}, \cite{Tza03}), and a Liouville-type result for isentropic solutions is obtained. In particular this implies the existence of a set $\br J$ (the candidate jump set) with $\mr{dim}_{\mathscr H}(\br J) \leq 1$, such that every point outside of $\br J$ is of vanishing mean oscillation. Let us point out that the $1$-rectifiability of $\br J$ is an open problem at the time of writing.

The aim of this paper is to further exploit the kinetic formulation of \cite{AMT25} and prove some qualitative properties of solutions.

\subsection{Strong Time regularity} Our first and main contribution is to show that entropy solutions (or more generally finite entropy solutions, see Proposition \ref{prop:e-fe}) are strongly time continuous.
\begin{thm}\label{thm:strongtime}
     Let $\bs u:  [0, +\infty) \times \mathbb R \to \mc U$ be a finite entropy solution to \eqref{eq:systemi} (see Definition \ref{defi:fes}). Assume that the system is strictly hyperbolic and genuinely nonlinear. Then $\bs u \in C^0([0, +\infty), \mathbf L^1_{loc}(\mathbb R))$.
\end{thm}
By Proposition \ref{prop:e-fe}, if there exists a uniformly convex entropy $E : \mc U \to \mathbb R$, entropy solutions are in particular also finite entropy solutions, so that the result applies.
To the best of our knowledge, the only available result concerning time regularity for $2\times 2$ systems of conservation laws in the $\mathbf L^\infty$
  setting is due to \cite{Vas99}. There, the author considers the isentropic Euler system with $\gamma = 3$, and establishes time regularity by leveraging the specific structure of the system through novel blow-up techniques. We point out that Theorem~\ref{thm:strongtime} also applies to entropy solutions of the isentropic Euler system for a general exponent \( \gamma \), provided that the solutions remain uniformly away from vacuum (i.e., \( \rho > c > 0 \)) (however this assumption was not needed in \cite{Vas99}). This condition ensures the uniform validity of the hyperbolicity assumption~\eqref{eq:hyperb}, which is used throughout the paper.

\begin{remark}
The existence of a uniformly convex entropy is a standard assumption which is satisfied in all physical cases. Moreover, under mild conditions on the Riemann invariants (see \cite[Section 12.2]{Daf16}), the Lax entropies, which are entropies of the following form:
$$
\eta^1_k(w, z) := e^{kw}\left[\phi(w,z) + \frac{1}{k}\chi(w,z) + \mc O(1/k^2)\right]
$$
$$
\eta^2_k(w, z) := e^{kz}\left[\alpha(w,z) + \frac{1}{k}\beta(w,z) + \mc O(1/k^2)\right]
$$
are uniformly convex for big enough $k$.
\end{remark}

 The proof of Theorem \ref{thm:strongtime} is based on some Lagrangian techniques, first introduced in the scalar case \cite{BBM17, Mar19}, and adapted to systems of conservation laws in \cite{AMT25}, and on blow up techniques, introduced for the first time in \cite{Vas99} (see also \cite{DOW03})

\subsection{A decay property for vanishing viscosity solutions} Our second result, which is of a different nature but whose proof is still related to the kinetic formulation, is the following dispersive-type estimate. 

\begin{thm}\label{thm:decay}
    Assume that the system \eqref{eq:systemi} is genuinely nonlinear, strictly hyperbolic, and that there exists a uniformly convex entropy $E : \mc U \to \mathbb R$. Let $\bs u$ be a vanishing viscosity solution to \eqref{eq:systemi} (see Definition \ref{defi:vvs}). Assume that for some $\widehat{\bs u} \in \mc U$ we have $\bs u(0, x) - \widehat{\bs u} \in \mathbf L^1$. Then there is a constant $C$ depending only on $f$ and $E$ such that 
    \begin{equation}
        \int_0^{+\infty} \int_{\mathbb R} |\bs u(t,x) - \widehat{\bs u}|^4 \dif x \dif t \leq C \Big( \|\bs u(0, \cdot)- \widehat{\bs u}\|_{\mathbf L^1} +\|\bs u(0, \cdot)- \widehat{\bs u}\|_{\mathbf L^1}^2 \Big).
    \end{equation}
\end{thm}

The result is restricted to vanishing viscosity solutions, as these exhibit additional structure in the associated dissipation measures—a property that is not a priori guaranteed for general entropy solutions (see Theorem \ref{thm:kin}).

Similar dispersive estimates have been derived for the Euler equations of gas dynamics in \cite{LPT94b}, where a kinetic formulation was introduced and used to establish decay estimate for weak entropy solutions.
Other estimates of dispersive nature, analogous in spirit to the ones we consider here, were established for the multidimensional Burgers equation by employing the framework of compensated integrability, as developed in \cite{Ser18a,Ser19}.  There, it was shown that initial data belonging to $\mathbf{L}^1$ are instantaneously regularized to higher integrability spaces $\mathbf{L}^p$ for all positive times. Here that framework cannot be directly applied because the kinetic formulation available in the setting of $2\times 2$ systems \cite{AMT25} does not involve positive kinetic functions, i.e. they can change sign. In \cite{SS19} the authors managed to obtain an even stronger decay, deriving $\mathbf{L}^\infty$ estimates. We remark that a stronger decay is to be expected for systems of  $2\times 2$ conservation laws, although establishing it appears to be out of reach with the methods introduced in this paper. In particular, it remains an open problem to establish a stronger $\mathbf L^\infty$ decay, analogously to what happens in the small total variation setting.

\vspace{1cm}
The paper is organized as follows. In Section \ref{sec:preliminaries} we outline the general setting and we recall some preliminary results. In Section \ref{sec:strongtime} we prove Theorem \ref{thm:strongtime}. Finally, in Section \ref{sec:decay} we prove Theorem \ref{thm:decay}.

\section{Preliminary Results}\label{sec:preliminaries}
\subsection{General Setting}
 Throughout the paper, we assume that the system \eqref{eq:systemi} is strictly hyperbolic in $\mc U$, so that the matrix $\mr D f$ has distinct real eigenvalues 
\begin{equation}\label{eq:hyperb}
\lambda_1(\bs u) < \lambda_2(\bs u) \qquad \forall \; \bs u \in \mc U
\end{equation}
with corresponding eigenvectors $r_1(\bs u), r_2(\bs u)$.
We also let $\ell_1, \ell_2$ be the corresponding left eigenvectors, normalized so that 
$$
\ell_i(\bs u) \cdot r_i(\bs u) = \delta_{i,j} \qquad \forall \; \bs u \in \mc U.
$$

    Being a system of two equations, \eqref{eq:systemi} admits a coordinate system of \textit{Riemann invariants} $\phi_1, \phi_2$. The latter are smooth invertible functions $\phi = (\phi_1, \phi_2 ): \mc U \to \mathbb R^2$ defined by 
\begin{equation}\label{eq:riemdef}
    \nabla \phi_1(\bs u) = \ell_1(\bs u), \qquad \nabla \phi_2(\bs u) = \ell_2(\bs u) \qquad \forall \bs u \in \mc U.
\end{equation}
We let $\mc W \doteq (\phi_1, \phi_2)(\mc U) \subset \mathbb R^2$.  Since we consider bounded solutions, here $\mc U$ is assumed to be a rectangle in Riemann invariants, i.e. there holds
$$
\mc W = [\ubar w, \bar w] \times [\ubar z, \bar z].
$$
A function \(g\) can be expressed in terms of the state vector \(\bs u\) or in terms of the Riemann invariants \((w, z)\), according to 
\[
g(\bs u) = \hat g(\phi^{-1}(\bs u)), \qquad
\partial_w \hat g = r_1 \cdot \nabla g, \quad \partial_z \hat g = r_2 \cdot \nabla g.
\]
From now on, relying on a common abuse of notation, we will use the same symbol \(g\) for both expressions. 
    A pair of Lipschitz functions $\eta, q: \mc U \subset \mathbb R^2 \to \mathbb R$ is called an \textit{entropy-entropy flux pair} for \eqref{eq:systemi} if
    \begin{equation}\label{eq:entropyeq}
        \nabla \eta(\bs u) \cdot \mr D f(\bs u) = \nabla q(\bs u) \qquad \text{for almost every $\bs u \in \mc U$}.
    \end{equation}
 Admissible (entropy) solutions of \eqref{eq:systemi} are the ones that dissipate the family of \textit{convex} entropies. 

\begin{defi}
    A  function $\bs u: \Omega \to \mc U$ is called an \textit{entropy weak solution} of \eqref{eq:systemi} if it satisfies 
 \begin{equation}
     \mu_{\eta}:= \partial_t \eta(\bs u) + \partial_x q(\bs u) \leq 0 \qquad \text{in} \; \msc D^{\prime}(\Omega)
 \end{equation}
 for all entropy-entropy flux pairs $(\eta, q)$ with $\eta$ a convex function.
 \end{defi}
Given an open set $A \subset \mathbb R^n$, we denote by $\msc M(A)$ the space of Radon measures in $A$. A more general notion of solution is the following.
\begin{defi}\label{defi:fes}
We say that a weak solution of \eqref{eq:systemi} $\bs u$ is a \emph{finite entropy solution} to \eqref{eq:systemi} if $\mu_{\eta} \in \msc M(\Omega)$ for every entropy-entropy flux pair $(\eta, q)$ of class $C^2$. Moreover, $\bs u$ is \emph{isentropic} if $\mu_\eta = 0$ for every $(\eta, q)$ in the same class.
\end{defi}
Notice that here a sign on $\mu_\eta$ is not imposed, and the convexity of $\eta$ is not required. It is a simple but crucial observation that, in the presence of a uniformly convex entropy $E : \mc U \to \mathbb R$, entropy solutions are also finite entropy solutions.
\begin{prop}\label{prop:e-fe}
    Let $\bs u: \Omega \subset \mathbb R^+ \times \mathbb R \to \mc U$ be an entropy solution, and assume that there exists a uniformly convex entropy $E : \mc U \to \mathbb R$. Then $\bs u$ is also a finite entropy solution.
\end{prop}
\begin{proof}
    Let $\eta$ be any $C^2$ entropy, and let $\kappa > 0$ big enough such that $\eta + \kappa E$ is convex. Then we have
    $$
    \begin{aligned}
        \partial_t E(\bs u) + \partial_x G(\bs u) = \mu_E,\\
        \partial_t (\eta + \kappa E)(\bs u) + \partial_x (q + \kappa G)(\bs u) = \mu_{\eta + \kappa E}
    \end{aligned}
    $$
    where $G$ is the entropy flux of $E$, and $\mu_E$, $\mu_{\eta + \kappa E}$ are locally finite negative measures, because $\bs u$ is entropic. Then 
    $$
    \partial_t \eta(\bs u) + \partial_x q(\bs u) = \mu_{\eta + \kappa E} - \kappa \mu_E
    $$
    which proves the result.
\end{proof}

As it is well known, the viscous approximations to \eqref{eq:systemi}
 \begin{equation}\label{eq:vsystem}
 \partial_t \bs u^\epsilon(t,x) + \partial_x f(\bs u^\epsilon(t,x)) = \epsilon \partial^2_{xx} \bs u^\epsilon,  \qquad (t,x) \in \Omega \subset \mathbb R^+ \times \mathbb R   , \qquad \bs u \in \mc U
 \end{equation}
 produce entropy admissible weak solutions of \eqref{eq:systemi} in the limit $\eps \to 0^+$. 

 \begin{defi}\label{defi:vvs}
      We say that $\bs u: \Omega \to \mc U$ is a \emph{vanishing viscosity solution} to \eqref{eq:systemi} if there exists a sequence $\eps_i \to 0^+$ and solutions $\bs u^{\eps_i}: \Omega \to \mc U$  to \eqref{eq:vsystem} such that ${\bs u}^{\eps_i} \to \bs u$ in  $\mathbf L^1_{\mr{loc}}(\Omega)$.
 \end{defi}

We conclude by recalling the definition of \emph{genuine nonlinearity}.

\begin{defi}\label{defi:GNL}
We say that the eigenvalue $\lambda_1$ ($\lambda_2$) is genuinely nonlinear (GNL) if there is $\bar c > 0$ such that 
$$
\partial_w\lambda_{1}(\bs u)\geq \bar c \quad \Big( \partial_z\lambda_{2}(\bs u) \geq \bar c\Big)\qquad \forall \; \bs u \in \mc U.
$$
\end{defi}

\subsection{Kinetic formulation}

We recall the construction in \cite{AMT25} that leads to a kinetic formulation. It is based on the singular entropies of \cite{PT00}, \cite{Tza03}.
Let $g, h$ be 
$$
\begin{aligned}
g(w, z) & = \exp\left[\int_{{\ubar z}}^{ z} -\frac{\lambda_{1z}(w, y)}{\lambda_1(w, y)-\lambda_2(w, y)}\dif y  \right] \\
h(w, z) & = \exp\left[\int_{{\ubar w}}^{w} \frac{\lambda_{2w}(y, z)}{\lambda_1(y, z)-\lambda_2(y, z)}\dif y  \right].
\end{aligned}
$$
\begin{defi}
We denote by $\bs \Theta[\xi]$ the entropy constructed as the unique solution to the Goursat-boundary value problem 
$$
\begin{cases}
\bs \Theta_{wz} = \frac{g_z}{g}\bs \Theta_w + \frac{h_w}{h} \bs \Theta_z, & \text{in} \; \mc W\\
\\
\bs \Theta(w, {\ubar z}) = 1, & \forall \; w \in [{\ubar w}, \bar w] \\
\\
\bs \Theta(\xi, z) = g(\xi, z) & \forall \; z \in [{\ubar z}, \bar z].
\end{cases}
$$
\end{defi}
We call $\bs \Xi$ the entropy flux of $\bs \Theta$.
Up to an additive constant in the entropy flux we can assume that (see \cite{AMT25} or \cite{PT00})
\begin{equation}\label{eq:efluxxi}
\bs \Xi(\xi, z) = \lambda_1(\xi, z) \bs \Theta(\xi, z) \qquad \forall \; z \in [{\ubar z}, \bar z].
\end{equation}
We define distributional entropies (see \cite{PT00})
\begin{equation}\label{e:chir}
\bs \chi[\xi]  \doteq \bs \Theta[\xi]\cdot  \mathbf 1_{\{w \geq \xi\}} , \qquad \wt{\bs \chi}[\xi]  \doteq \bs \Theta[\xi]\cdot  \mathbf 1_{\{w \leq \xi\}} 
\end{equation}
and entropy fluxes
\begin{equation}\label{eq:fluxes}
\bs \psi[\xi] \doteq \bs \Xi[\xi]\cdot  \mathbf 1_{\{w \geq \xi\}} \qquad \wt{\bs \psi}[\xi] \doteq \bs \Xi[\xi]\cdot  \mathbf 1_{\{w \leq \xi\}}.
\end{equation}

For a proof of the following Proposition see e.g. \cite{AMT25}.
\begin{prop}\label{prop:localspeed}
    There exists positive $\bar {r}, c > 0$  such that, for every $\xi, w \in [\ubar w, \bar w]$  and $z \in [\ubar z, \bar z]$ such that $\xi \leq w \leq \xi + \bar r$, the following holds:
\begin{enumerate}
    \item Local strict positivity of the entropies: 
    $$
    \bs \chi[\xi](w, z) \geq c > 0
    $$
    \item If $\lambda_{1}$ is genuinely nonlinear, then we have the monotonicity of the kinetic speed:
$$
\frac{\dif}{\dif \xi} \bs \lambda_1[\xi](w, z) \geq c > 0 
$$
where
\begin{equation}\label{eq:lambdadef}
\bs \lambda_1[\xi](w, z) \doteq \frac{\bs \psi[\xi](w, z)}{\bs \chi[\xi](w, z)} \qquad   \forall \; \xi \leq w \leq \xi  +\bar r.
\end{equation}
\end{enumerate}
A completely symmetric statement holds for the entropies $\wt{\bs \chi}$.
\end{prop}

A symmetric construction can be made for the entropies that can be cut along the second Riemann invariant; for these entropies, for $\zeta \in [{\ubar z}, \bar z]$, we let $\bs \upsilon[\zeta](w,z)$ be entropy corresponding to $\bs \chi[\xi](w,z)$, and $\bs \varphi[\zeta](w,z)$ for the respective entropy flux, corresponding to $\bs \psi[\xi](w,z)$.

In the following, given a function $\bs u: \mathbb R^+ \times \mathbb R \to \mc U$, we define 
\begin{equation}\label{eq:kfdef}
\begin{aligned}
    & \bs \chi_{\bs u}(t,x,\xi) \doteq \bs \chi[\xi](\bs {\bs u}(t,x)) \qquad   \forall \; (t, x, \xi) \in \Omega \times (\ubar w, \bar w), \\
    & \bs \upsilon_{\bs u}(t,x,\zeta) \doteq \bs \upsilon[\zeta](\bs {\bs u}(t,x)) \qquad   \forall \; (t, x, \zeta) \in \Omega \times (\ubar z, \bar z)
    \end{aligned}
\end{equation}
\begin{equation}\label{eq:kfdef1}
\begin{aligned}
    & \bs \psi_{\bs u}(t,x,\xi) \doteq \bs \psi[\xi](\bs {\bs u}(t,x)) \qquad   \forall \; (t, x, \xi) \in \Omega \times (\ubar w, \bar w), \\
    & \bs \varphi_{\bs u}(t,x,\zeta) \doteq \bs \varphi[\zeta](\bs {\bs u}(t,x)) \qquad   \forall \; (t, x, \zeta) \in \Omega \times (\ubar z, \bar z).
    \end{aligned}
\end{equation}

Notice that the support of $\bs \chi_{\bs u}, \bs \psi_{\bs u}$ is contained in the hypograph of the first Riemann invariant $w = \phi_1(\bs u)$, while the support of $\bs \upsilon_{\bs u}$, $\bs \varphi_{\bs u}$ is contained in the hypograph of the second Riemann invariant $z = \phi_2(\bs u)$.

\begin{prop}[\cite{AMT25}]\label{prop:kfeseq}
   Let $\Omega \subset \mathbb R^2$ be an open set and let $\bs u \in \mathbf L^\infty(\Omega, \, \mc U)$. Then $\bs u$ is a finite entropy solution to \eqref{eq:systemi} if and only if there are locally finite measures $\mu_1, \mu_0 \in \msc M(\Omega \times (\ubar w, \bar w))$ and $\nu_1, \nu_0 \in \msc M(\Omega \times (\ubar z, \bar z))$  such that 
    \begin{equation}\label{eq:kin22}
        \; \; \partial_t \bs \chi_{\bs u}(t,x,\xi) + \partial_x \bs \psi_{\bs u}(t,x,\xi) = \partial_{\xi} \mu_1 + \mu_0 \qquad \text{in $\msc D^\prime$}\big(\Omega \times ({\ubar w}, \bar w) \big)
    \end{equation}
     \begin{equation}\label{eq:kin221}
        \partial_t \bs \upsilon_{\bs u}(t,x,\zeta) + \partial_x \bs \varphi_{\bs u}(t,x,\zeta) = \partial_\zeta \nu_1 + \nu_0 \qquad \text{in $\msc D^\prime$}\big(\Omega \times ({\ubar z}, \bar z) \big)
    \end{equation}
   Moreover, $\bs u$ is an isentropic solution, i.e. for every smooth entropy pair $\eta, q$ it holds $\mu_\eta = 0$, if and only if 
    \begin{equation}\label{eq:kin22se}
        \partial_t \bs \chi_{\bs u}(t,x,\xi) + \partial_x \bs \psi_{\bs u}(t,x,\xi) = 0 \qquad \text{in $\msc D^\prime$}\big(\Omega \times \mathbb R \big)
    \end{equation}
     \begin{equation}\label{eq:kin221se}
        \partial_t \bs \upsilon_{\bs u}(t, x, \zeta)+ \partial_x \bs \varphi_{\bs u}(t, x, \zeta) =0 \qquad \text{in $\msc D^\prime$}\big(\Omega \times  \mathbb R \big)
    \end{equation}
\end{prop}
\begin{remark}\label{rem:prop:kfeseqsym}
Defining
\begin{equation}\label{eq:kfdeftilde}
\begin{aligned}
    & \widetilde{\bs \chi}_{\bs u}(t,x,\xi) \doteq \widetilde{\bs \chi}[\xi](\bs u(t,x)), \quad  \widetilde{\bs \psi}_{\bs u}(t,x,\xi) \doteq \widetilde{\bs \psi}[\zeta](\bs u(t,x)), \quad   \forall \; (t, x, \xi) \in \Omega \times (\ubar w, \bar w)
    \end{aligned}
\end{equation}
and with obvious notation also $\widetilde{ \bs \upsilon}_{\bs u}$, $\widetilde{ \bs \varphi}_{\bs u}$
one can make a symmetric statement: in particular,
$\bs u$ is an isentropic solution if and only if (recall \eqref{eq:kfdeftilde})
    \begin{equation}
        \partial_t \widetilde{\bs \chi}_{\bs u}(t,x,\xi) + \partial_x \widetilde{\bs \psi}_{\bs u}(t,x,\xi) = 0 \qquad \text{in $\msc D^\prime$}\big(\Omega \times \mathbb R \big)
    \end{equation}
     \begin{equation}
        \partial_t \widetilde{\bs \upsilon}_{\bs u}(t, x, \zeta)+ \partial_x \widetilde{\bs \varphi}_{\bs u}(t, x, \zeta) =0 \qquad \text{in $\msc D^\prime$}\big(\Omega \times  \mathbb R \big)
    \end{equation}
\end{remark}

The measures $\mu_i$, $\nu_i$ are related to the dissipation of smooth entropies. In fact, one can show that if
\begin{equation}\label{eq:supmueta}
\bs \nu \doteq \bigvee_{\substack{\eta \in \mc E \\ |\eta|_{C^2} < 1}} \mu_{\eta} \; \in \; \msc M(\Omega).
\end{equation}
then $|\mu_i| \leq \bs \nu$ and $|\nu_i| \leq \bs \nu$, provided that $\mu_i, \nu_i$ are chosen so that there are not cancellations in the right hand side \eqref{eq:kin22}, \eqref{eq:kin221}.
Here $\bigvee$ denotes the supremum in the sense of measures (see \cite[Definition 1.68]{AFP00}) and $\mc E$ is the set of entropies $\eta : \mc U \to \mathbb R$ of class $C^2$, while $\mu_\eta$ is the  dissipation measure 
$$
\mu_\eta := \eta_t(\bs u) + q_x(\bs u).
$$

\begin{remark}
    We remark that the kinetic equations \eqref{eq:kin22}, \eqref{eq:kin221} are not a characterization of entropy solutions, unlike the kinetic equation that can be obtained in the scalar case \cite{LPT94a}, in the case of the Euler system of gas dynamics \cite{LPT94b} or elastodynamics \cite{PT00}. They are a characterization of the more general concept of finite entropy solution, by Proposition \ref{prop:kfeseq}.
\end{remark}

\section{Strong Time Regularity}\label{sec:strongtime}
In this section we prove Theorem \ref{thm:strongtime}.  We need some preliminary tools from \cite{AMT25}, which we recall in the following Subsection.

\subsection{Lagrangian tools}
Let $\bs u$ be an isentropic solution to \eqref{eq:systemi} defined in $\mathbb R^+ \times \mathbb R$ and let $w = \phi_1(\bs u)$. 
From now on, all the results will be stated for the first Riemann invariant $w = \phi_1(\bs u)$, but it is clear that symmetric statements hold for the other $z = \phi_2(\bs u)$.
Define $w_{\max} \doteq \esssup_{t,x}  w$ and $w_{\min} \doteq \essinf_{t,x} w.$
Following \cite[Section 4]{AMT25}, for $r > 0$ sufficiently small we define
\begin{equation}\label{eq:chibcdef}
\begin{aligned}
   &  \bs \chi^{\max}(t,x,\xi) \doteq \bs \chi_{\bs u}(t,x,\xi)\cdot \mathbf 1_{\{(t,x,\xi) \; | \; w_{\max}- r \leq \xi \leq  w_{\max}\}}(t,x,\xi) \geq 0
       \end{aligned}
\end{equation}
\begin{equation}
    \begin{aligned}
     & \bs \chi^{\min}(t,x,\xi) \doteq \bs {\wt\chi}_{\bs u}(t,x,\xi)\cdot \mathbf 1_{\{(t,x,\xi) \; | \; w_{\min}  \leq \xi \leq  w_{\min} +r\}}(t,x,\xi) \geq 0
    \end{aligned}
\end{equation}
We have a strict inequality $\bs \chi^{\max}(t,x,\xi) > c> 0$ and $\bs \chi^{\min}(t,x,\xi) > c> 0$ when $w_{\max} -r < \xi < w(t,x)$ and $w(t,x) < \xi < w_{\min} + r $ respectively, which is an immediate consequence of Proposition \ref{prop:localspeed} if we choose $r < \bar r$. Notice that 
$$
\mr{supp} \, \bs \chi^{\max} = \mr{hyp} \,  \phi_1(\bs u) \cap \Big(\mathbb R^+ \times \mathbb R \times (w^{\max}-r, w^{\max}) \Big)
$$
and similarly for $\bs \chi^{\min}$
$$
\mr{supp} \, \bs \chi^{\min} = \mr{epi} \,  \phi_1(\bs u) \cap \Big(\mathbb R^+ \times \mathbb R \times (w^{\min}, w^{\min}+r) \Big)
$$
where $\mr{hyp} \, \phi_1(\bs u)$ and $\mr{epi}\, \phi_1(\bs u)$ denote the hypograph and epigraph, respectively, of the function $\phi_1(\bs u)$.

We have by  Proposition \ref{prop:localspeed} and Proposition \ref{prop:kfeseq}, that
\begin{equation}\label{eq:chikin}
 \partial_t \bs \chi^{\max} + \partial_x( \bs \lambda_1[\xi](\bs u(t,x)) \bs \chi^{\max}(t,x,\xi)) = 0 \qquad \text{in $\msc D^\prime(\mathbb R^+ \times \mathbb R \times  \mathbb R)$}
\end{equation}
\begin{equation}
     \partial_t \bs \chi^{\min} + \partial_x (\bs \lambda_1[\xi](\bs u(t,x)) \bs \chi^{\min}(t,x,\xi))= 0 \qquad \text{in $\msc D^\prime(\mathbb R^+ \times \mathbb R \times  \mathbb R)$}.
\end{equation}
where $\bs \lambda_1[\xi](\bs u)$ is as in \eqref{eq:lambdadef}.
 We have the following result from \cite{AMT25}:

\begin{prop}\label{prop:lagr}
Let $\bs \chi^{\max}$ be as above, satisfying equation \eqref{eq:chikin}. Then there exists a positive measure $\bs \omega \in \msc M^+(\Gamma)$, where
$$
\Gamma = \Big\{\gamma = (\gamma_x, \gamma_\xi) : \mathbb R^+ \to \mathbb R^2, \qquad \gamma_x \quad \text{Lipschitz curve}, \qquad \gamma_\xi \in \mathbf L^\infty(\mathbb R^+) \Big\}
$$
such that 
\begin{enumerate}
    \item $\bs \omega$ is concentrated on curves $\gamma \in \Gamma$ such that
    \begin{enumerate}
    \item[$(a_1)$] $\gamma_\xi$ is a constant function $\gamma_{\xi}(t) \equiv \xi_{\gamma} \in \mathbb R$ for all $t \in \mathbb R^+$;
        \item[$(b_1)$] $\gamma_x$ is characteristic for $\bs \lambda_1[\xi](\bs u)$:
    \begin{equation}\label{eq:charomega}
    \dot \gamma_x(t) = \bs \lambda_1[\xi_{\gamma}](\bs u(t,x)) \qquad \text{for a.e. $t \in \mathbb R^+$}.
        \end{equation}
         \end{enumerate}
    \item Up to redefining $\bs \chi^{\max}$  on a set of times of measure zero, we can recover it by superposition of the curves:
    \begin{equation}\label{eq:chisuper}
\bs \chi^{\max}(t, \cdot, \cdot)\cdot \msc L^2 = (e_t)_\sharp \bs \omega  \qquad \text{for all $t \in \mathbb R^+$}.
        \end{equation}
    where $e_t : \Gamma \to \mathbb R^2$ is the evaluation map $e_t(\gamma) = \gamma(t)$.
    \item 
     For $\bs \omega$ almost every $\gamma = (\gamma_x, \xi_{\gamma}) \in\Gamma$, for $\msc L^1$-almost every $t \in \mathbb R^+$ it holds
    \begin{enumerate}
        \item[$(a_2)$] $(t, \gamma_x(t))$ is a Lebesgue point of $\bs u$;
        \item[$(b_2)$] it holds $w(t,\gamma_x(t))-r \leq \xi_{\gamma} \leq w(t,\gamma_x(t))$.
    \end{enumerate}
  
\end{enumerate}

Entirely similar considerations hold for $\bs \chi^{\min}$; we call $\bs \eta \in \msc M^+(\Gamma)$ the corresponding measure. In particular (3) is replaced by 
 \begin{enumerate}
        \item[(3')] For $\bs \eta$ almost every $\sigma = (\sigma_x, \xi_\sigma) \in \Gamma$, for $\mathscr L^1$-a.e. $t \in \mathbb R^+$, it holds
        \begin{enumerate}
            \item[($a_2^{\text{'}}$)] $(t,\sigma_x(t))$ is a Lebesgue point of $\bs u$,
        \item[($b_2^{\text{'}}$)]  it holds $w(t,\sigma_x(t)) \leq \xi_{\sigma} \leq w(t,\sigma_x(t))+r$.
         \end{enumerate}
    \end{enumerate}
\end{prop}

\begin{defi}
      We let $\Gamma^{\max}$ be the set of curves where $\bs \omega$ is concentrated (i.e. those satisfying $(a_1), (b_1)$), that additionally satisfy $(a_2), (b_2)$. Similarly, we let $\Gamma^{\min}$ be a set of curves where $\bs \eta$ concentrated (i.e. satisfying $(a_1), (b_1)$, satisfying also $(a_2^{\text{'}}), (b_2^{\text{'}})$.
\end{defi}

\begin{remark}\label{rem:suppwchi}
We remark that from \eqref{eq:chisuper} the Riemann invariant $w$ and $\bs \omega$ are related by the following: for a.e. $(t,x)$ such that $w(t,x) > w_{\max}-r$, there holds
$$
w(t,x) = \sup \{\xi_\gamma \; | \; \gamma \in \Gamma^{\max}, \quad \gamma_x(t) = x\}
$$
and similarly for the other Riemann invariant.
\end{remark}

We refer to \cite{AMT25} for more details and the proof of the Proposition, let us only mention that it follows from the application of the Ambrosio superposition principle for the continuity equation \cite{Amb08}.

\begin{remark}\label{rem:loclagr}
    By a localization argument, the Lagrangian representation $\bs \omega$ can be obtained for isentropic solutions in a localized domain
    $$
    \Omega := \{(t,x) \; | \; y_1(t) < x < y_2(t)\}
    $$
    where $y_1< y_2: \mathbb R^+ \to \mathbb R$ are Lipschitz curves,
    provided that $\bs \chi^{\max}$
    has no flux through the boundaries $y_1, y_2$, i.e. that 
    \begin{equation}\label{eq:localsol}
        \int \int_{\overline \Omega} \varphi_t \bs \chi^{\max} + \varphi_{x}  \bs \lambda_1[\xi](\bs u) \bs \chi^{\max} \dif t \dif x \dif \xi = 0 \qquad \text{for all $\varphi \in C^1_c(\overline \Omega\times \mathbb R)$}, \quad \mr{supp} \, \varphi \subset \{t > 0\}.
    \end{equation}
    In fact, \eqref{eq:localsol} implies that the extended function 
    $$
    \widehat{\bs \chi}^{\max}(t,x, \xi) := \begin{cases}
        \bs \chi^{\max}(t,x, \xi) & \text{if $(t,x) \in \Omega$}\\
        0 & \text{outside}
    \end{cases}
    $$
    solves \eqref{eq:chikin} in the whole $\mathbb R^+\times \mathbb R^2$, so that Proposition \ref{prop:lagr} can be applied. It is then immediate to see that in this case Proposition \ref{prop:lagr} holds with
    $$
\Gamma = \Big\{\gamma = (\gamma_x, \gamma_\xi) : \mathbb R^+ \to \mathbb R^2, \quad \gamma_x \quad \text{Lipschitz}, \quad \gamma_\xi \in \mathbf L^\infty(\mathbb R^+), \quad y_1(t) \leq \gamma_x(t) \leq y_2(t)\Big\}.
$$
The same holds for $\bs \eta$, $\bs \chi^{\min}$.
\end{remark}

The following proposition is an immediate consequence of a result in \cite{AMT25}, we provide its proof below.

\begin{prop}\label{prop:inter}
 Let $\bs u :  \Omega \to \mc U$ be an isentropic solution, where 
 $$
 \Omega = \{(t,x) \; | \;  x \leq y_2(t)\}, \qquad y_2: \mathbb R^+ \to \mathbb R \quad \text{Lipschitz}.
 $$
Then there does not exist $r > 0$, with $r < \bar r$, such that
$$
\bs \chi^{\max}(t,x,\xi) := \bs \chi_{\bs u}(t,x,\xi) \cdot \mathbf 1_{\{(t,x,\xi) \; | \; w_{\max}- r \leq \xi \leq  w_{\max}\}}(t,x,\xi)
$$
 satisfies \eqref{eq:localsol}, where
$$
w_{\max} := \esssup_{t,x \in \Omega} \phi_1(\bs u).
$$ 
\end{prop}

\begin{proof}
   By contradiction, assume that there exists such an $r > 0$. Then, by Remark \ref{rem:loclagr}, we can consider its Lagrangian representation $\bs \omega$, with corresponding set of curves $\Gamma^{\max}$. Next, let $\bar \gamma$ be a curve in $\Gamma^{\max}$ such that
\begin{equation}\label{eq:contrhyp}
 b \doteq \xi_{\bar \gamma} > w_{\max}-r  \doteq a 
\end{equation}
Define
\begin{equation}\label{eq:Qdef}
    \mc Q(t) \doteq \int_{a}^{b} \int_{\bar \gamma_x(t)}^{y_2(t)} \bs \chi^{\max}(t,x,\xi) \dif x \dif  \xi, \qquad t \geq 0.
\end{equation}
Then, proceeding exactly as in \cite{AMT25}, Proposition 5.1, we deduce that there is $C>0$ such that for every $t > 0$ it holds
\begin{equation}
    \mc Q(t)  - \mc Q(0) \leq -t \, C 
\end{equation}
This yields a contradiction since $\mc Q(0)$ is finite and $\mc Q(t)$ is positive for all $t > 0$.
\end{proof}

\subsection{Proof of the time regularity} We are now ready to prove Theorem \ref{thm:strongtime}. We will use some preliminary results.
 The next Lemma will be used later (for the proof see e.g. \cite[Lemma 5]{Mar22}).
\begin{lemma}\label{lemma:traces}
Assume that $\bar \sigma_x:(t_1,t_2)\subset \mathbb R\to \R$ is a Lipschitz curve and that for $\msc L^1$-a.e. $t \in (t_1,t_2)$ the point $(t, \bar \sigma_x(t))$ is a Lebesgue point of $w \in \mathbf L^\infty(\mathbb R^2; \,\mathbb R)$. Then
$$
\lim_{\delta \to 0} \int_{t_1}^{t_2} \frac{1}{\delta}\int_{\bar \sigma_x(t)-\delta}^{\bar \sigma_x(t)+\delta}|w(t,x)-w(t,\bar \sigma_x(t))|\dif x \dif t = 0.
$$
\end{lemma}

In the scalar case $f(u) = u^2$, if a solution $u: \mathbb R^+ \times \mathbb R \to \mathbb R$ is isentropic, then $u(t, \cdot)$ must be an increasing function for all $t >0$. The following Lemma is a weaker property of this type which remains valid in the setting of $\mathbf L^\infty$ solutions to genuinely nonlinear $2\times 2$ systems.
\begin{lemma}\label{lemma:lrsol}
    Let $\bs u: \mathbb R^+ \times \mathbb R \to \mc U$ be an isentropic solution and let $w = \phi_1(\bs u)$.
    Then for every $T > 0$
\begin{equation}\label{eq:-inf}
\lim_{M \to -\infty} \Bigg( \esssup_{(t,x) \in (0,T)\times (-\infty, M) } w \Bigg) \; = \;  w_{\min}
\end{equation}
and
\begin{equation}\label{eq:+inf}
\lim_{M \to +\infty}\Bigg( \essinf_{(t,x) \in (0, T) \times (M, +\infty) } w \Bigg)\; =\; w_{\max}
\end{equation}
\end{lemma}

\begin{proof}
  Let $\Gamma^{\min}, \Gamma^{\max}$ be the set of curves of Proposition \ref{prop:lagr} for the solution $\bs u$, and $\bs \chi^{\max}$, $\bs \chi^{\min}$ defined as in \eqref{eq:chibcdef}. Take any $\bar \sigma \in \Gamma^{\min}$ and define
   $$
   \Omega_{\bar \sigma} := \{(t,x) \; | \; x \leq \bar \sigma_x(t)\}, \quad \bs v := \bs u_{\mid \Omega_{\bar \sigma}}.
   $$
 By contradiction, assume that
    \begin{equation}\label{eq:claimv}
         w^\ell_{\max}:= \esssup_{\Omega_{\bar \sigma}} \phi_1(\bs v) > \bar \sigma_{\xi}.
    \end{equation}   
Let 
\begin{equation}\label{eq:chielldef}
\bs \chi_{\ell}^{\max}(t, x, \xi) := \bs \chi_{\bs v}(t,x,\xi)\mathbf 1_{\{(t,x,\xi) \; | \;  w^\ell_{\max}- \wt r \leq \xi \leq   w^\ell_{\max}\}}
\end{equation}
where $\wt r$ is so small that $w^{\ell}_{\max} - \wt r > \bar \sigma_{\xi}$, where we recall that $\bs \chi_{\bs v}$ is just
$$
\bs \chi_{\bs v} = (\bs \chi_{\bs u})_{\mid \Omega_{\bar \sigma} }.
$$
Since  $\bar \sigma \in \Gamma^{\min}(\bs u)$, by (3') of Proposition \ref{prop:lagr} we obtain that $(t, \bar \sigma(t))$ is a Lebesgue point of $w$ for a.e. $t > 0$ and that
$$
w(t, \bar \sigma(t)) < \bar \sigma_\xi \qquad \text{for a.e. $t > 0$}.
$$

\vspace{0.3cm}
Next, we claim that $\bs \chi_{\ell}^{\max}$ has no flux through the lateral boundary of $\Omega_{\bar \sigma}$, in particular that it satisfies \eqref{eq:localsol}. This will yield a contradiction, by Proposition \ref{prop:inter}, thus proving the result.

To prove the claim, we first notice that since $\bs \chi_{\bs u}$ satisfies \eqref{eq:chikin}, then by definition \eqref{eq:chielldef} we have 
\begin{equation}\label{eq:chiell1}
    \int \int_{\Omega_{\bar \sigma}} \varphi_t  \bs \chi_{\ell}^{\max} + \varphi_x  \bs \lambda_1[\xi](\bs u) \bs \chi_{\ell}^{\max} \dif x \dif t \dif \xi = 0 \qquad \forall \; \varphi \in C^1_c(\Omega_{\bar \sigma}), \quad \mr{supp} \, \varphi \subset \{t > 0\}.
\end{equation}
To conclude the proof we need to extend \eqref{eq:chiell1}  to all $\varphi \in C^1_c({\overline \Omega}_{\bar \sigma})$ with $\mr{supp} \, \varphi \subset \{t > 0\}$. Let $\varphi$ be such a function and consider a smooth cutoff $h_\varepsilon$ of the following form: 

\begin{equation}
\begin{cases}
 h_\varepsilon(t,x) = 1 & \qquad \text{in $\{(t,x) \;  | \; x \leq \bar \sigma_x(t) -2\varepsilon\}$} \\
 h_\varepsilon(t,x) = 0 & \qquad \text{in $\{(t,x) \;  | \;  \bar \sigma_x(t) -\varepsilon < x <\bar \sigma_x(t)\}$}\\
 |h_\varepsilon(t,x)| < 1,  \quad  |\nabla h_{\varepsilon}| \leq \frac{C}{\varepsilon} & \qquad \text{everywhere}
\end{cases}
\end{equation}
where $C> 0$ is a constant independent of $\varepsilon$. This is possible because $\bar \sigma_x$ is Lipschitz. Define $\varphi^\varepsilon := \varphi h_\varepsilon$, so that $\varphi^\varepsilon$ can be now used in \eqref{eq:chiell1} and we get 
\begin{equation}
\begin{aligned}
     \lim_{\varepsilon \to 0} \Bigg[& \int_{\Omega_{\bar \sigma}} \int h_\varepsilon \Big( \varphi_t  \bs \chi_{\ell}^{\max}  + \varphi_x  \bs \lambda_1[\xi](\bs u) \bs \chi_{\ell}^{\max} \Big) \dif \xi \dif x \dif t  \\
    & + \int_{\Omega_{\bar \sigma}}\int \varphi \Big( h^\varepsilon_t  \bs \chi_{\ell}^{\max} + h^\varepsilon_x  \bs \lambda_1[\xi](\bs u) \bs \chi_{\ell}^{\max} \Big) \dif \xi \dif x \dif t \Bigg]= 0
\end{aligned}
\end{equation}
The second term can be estimated by 
\begin{equation}
    \begin{aligned}
         \left|\int_{\Omega_{\bar \sigma}}\int \varphi \Big( h^\varepsilon_t  \bs \chi_{\ell}^{\max} + h^\varepsilon_x  \bs \lambda_1[\xi](\bs u) \bs \chi_{\ell}^{\max} \Big) \dif \xi \dif x \dif t \right| & \leq C \int \frac{1}{\varepsilon}\int_0^T\int_{\bar \sigma_x(t)-2\varepsilon}^{\bar \sigma_x(t)-\varepsilon} \bs \chi_\ell^{\max}(t,x,\xi) \dif x \dif t \dif \xi
    \end{aligned}
\end{equation}
where $T> 0$ is big enough so that $\mr{supp} \; \varphi \subset (0,T) \times \mathbb R^2$. Since 
$$
\bs \chi_\ell^{\max}(t,x,\xi) \neq 0 \quad \Longrightarrow \quad w(t,x) > \xi > \bar \sigma_\xi \geq w(t,\bar \sigma_x(t)) \quad \text{for a.e. $(t,x, \xi)$}
$$
we deduce
$$
\begin{aligned}
    \frac{1}{\varepsilon}\int_0^T\int_{\bar \sigma_x(t)-2\varepsilon}^{\bar \sigma_x(t)-\varepsilon} \int \bs \chi_\ell^{\max}(t,x,\xi) \dif \xi \dif x \dif t &\leq C \frac{1}{\varepsilon}\int_0^T\int_{\bar \sigma_x(t)-2\varepsilon}^{\bar \sigma_x(t)-\varepsilon} |w(t,x)- w(t,\bar \sigma_x(t))| \dif x \dif t.
\end{aligned}
$$
Finally, the second term above goes to zero thanks to Lemma \ref{lemma:traces}, and this proves our claim.
This proves our claim.
\vspace{0.3cm}

By Remark \ref{rem:suppwchi}, we have thus proved that for every $\bar \sigma \in \Gamma^{\min}$, there holds
$$
\esssup \bs u_{\mid \{(t,x) \, | \; x \leq \bar \sigma_x(t)\}} \leq \bar \sigma_\xi
$$
Since $|\dot {\bar \sigma}_x| \leq L$ for some uniform positive constant $L >0$, we have
$$
\esssup \bs u_{\mid \{(t,x) \, | \; x \leq \bar \sigma_x(T)- LT, \quad t \in [0,T]\}} \leq \bar \sigma_\xi
$$
The proof is concluded by taking a sequence of curves $\bar \sigma^n \in \Gamma^{\min}$ such that 
$$
w_{\min} \leq \bar \sigma^n_\xi \leq w_{\min} + \frac{1}{n}
$$
whose existence is guaranteed by Proposition \ref{prop:lagr}.
\end{proof}

\begin{coro}\label{lemma:uconst}
    Let $\bs u : \mathbb R^+ \times \mathbb R \to \mc U$ be an isentropic solution and let
$$
\bar {\bs u}_0 := \text{weak$^*$}-\lim_{t \to 0^+} \bs u(t,\cdot)
$$ Assume that $\bar {\bs u}_0$ is constant.
    Then $\bs u(t,x) =  \bar {\bs u}_0$ for a.e. $(t,x)$, i.e. $\bs u$ is also constant in $\mathbb R^+ \times \mathbb R$.
\end{coro}

\begin{proof}
We prove that $w = \phi_1(\bs u)$ is constant, the proof for the other Riemann invariant being entirely symmetric.
By \eqref{eq:-inf}, \eqref{eq:+inf} it holds
$$
w_{\min} = \lim_{x \to -\infty} \bar {\bs u}_0(x) = \lim_{x \to +\infty} \bar {\bs u}_0(x) = w_{\max}
$$
and this proves the Corollary.
\end{proof}

The following is a standard Lemma that is useful for later.
\begin{lemma}\label{lemma:covering}
    Let $\bs \nu$ be a locally finite measure in $\mathbb R^+ \times \mathbb R$. Then for every $\bar t \geq 0$ there holds  \begin{equation}\label{eq:covlemma}
        \lim_{r \to 0^+} \frac{1}{r}\bs \nu \left( \big((\bar t-r, \bar t) \cup (\bar t , \bar t + r)\big) \times (\bar x- r, \bar x+ r)\right) = 0 \qquad \text{for a.e. $\bar x \in \mathbb R$}
    \end{equation}
    \end{lemma}

\begin{proof}
Let 
    $$
   H_{\bar t} = \{(t,x) \; | \;  t = \bar t\}, \qquad \mu = \mathscr H^1 \llcorner H_{\bar t}, \qquad \wt {\bs \nu} := \bs \nu - \bs \nu\llcorner H_{\bar t}.
    $$
    By the Besicovitch differentiation theorem (see for example \cite[Theorem 2.22]{AFP00}) we have
    $$
    \wt{\bs \nu} = f \mu + \wt{\bs \nu}^s
    $$
    where $f$ is defined $\mu$-a.e. by 
    $$
    f(t,x) := \lim_{r \to 0^+} \frac{\wt{\bs \nu}(B_r((t,x)))}{\mu(B_r((t,x)))} \qquad \text{for $(t,x) \in \mr{supp}\, \mu$}
    $$
    and $\wt{\bs \nu}^s$ is a positive measure (concentrated on a $\mu$-negligible set). Since $\wt{\bs \nu}(H_{\bar t}) = 0$, we must have $f(t,x) = 0$ for $\mu$-a.e. $(t, x$, that is the same of 
    $$
    \lim_{r \to 0^+}\frac{1}{2r}\bs \nu \left(\big((\bar t-r, \bar t) \cup (\bar t , \bar t + r)\big) \times (\bar x-r,\bar x+ r)\right) = 0 \qquad \text{for a.e. $\bar x \in \mathbb R$.}
    $$
\end{proof}

\begin{lemma}\label{lemma:timestrip}
    Let $\bs u : \mathbb R^+ \times \mathbb R$ be a finite entropy solution to \eqref{eq:systemi}. Let $\bar t \geq 0$ fixed. Then for every $M > 0$
    $$
    \lim_{r \to 0} \fint_0^r \int_{-M}^M| \bs u(\bar t + s, x)- \bs u(\bar t,x)| \dif x \dif s = 0
    $$
\end{lemma}

\begin{proof}
    \textbf{1.} In this step, we show that 
   \begin{equation}\label{eq:fullconv}
 \bs u_r^{\bar x} \longrightarrow \bs u(\bar t, \bar x) \qquad  \text{strongly in $\mathbf L^1_{loc}(\mathbb R^+\times \mathbb R)$ as $r \to 0^+$} \quad \forall \; \bar x \in \mathbb R \setminus \mc N
\end{equation}
where $\msc L^1(\mc N) = 0$, and for $\bar x \in \mathbb R \setminus \mc N$ we define the rescalings 
$$
\bs u^{\bar x}_r(s, y) := \bs u(\bar t + r s, \bar x + r y). 
$$
    
    From Lemma \ref{lemma:covering} applied with $\bs \nu$ as in \eqref{eq:supmueta}, we deduce
    \begin{equation}\label{eq:nufub}
        \lim_{r \to 0^+} \frac{1}{r}\bs \nu \left((\bar t , \bar t + r) \times (\bar x-r,\bar x+ r)\right) = 0 \qquad \text{for a.e. $\bar x \in \mathbb R$}
    \end{equation}

Let $\mc N \subset \mathbb R$ be a Lebesgue null set containing all the points at which \eqref{eq:nufub} fails and the set of non-Lebesgue points of $u(\bar t, \cdot)$.
Then, up to a subsequence $\{r_j\}$, by compensated compactness (see e.g. \cite[Chapter 9]{Ser00}, \cite{Tza03}, or \cite{AMT25} for a summary of the results) and by \eqref{eq:nufub}, there is a weak  isentropic solution $\bs u_\infty^{\bar x} : \mathbb R^+ \times \mathbb R \to \mc U$ such that 
\begin{equation}\label{eq:rjtouinf}
\bs u^{\bar x}_{r_j}  \longrightarrow  \bs u_{\infty}^{\bar x} \qquad \text{strongly in $\mathbf L^1_{loc}(\mathbb R^+ \times \mathbb R)$ as $j \to +\infty$}.
\end{equation}
We now claim that 
\begin{equation}\label{eq:constinit}
     \bs u^{\bar x}_{\infty}(t,\cdot)\longrightarrow \bs u(\bar t, \bar x) \qquad \text{weakly$^*$ in $\mathbf L^\infty(\mathbb R)$ as $t \to \bar t^+$}.
\end{equation}
In fact, take any $\phi \in C^1_c(\mathbb R)$ supported in a compact set $[-L, L]$. By \eqref{eq:rjtouinf} we have
$$
\int_{\mathbb R^+} \Phi_{r_j}(t) \dif t \longrightarrow 0 \quad \text{as $j \to +\infty$}, \quad \Phi_{r_j}(t) := \int_{\mathbb R} \phi(y)|\bs u_{\infty}(t,x) - \bs u_{r_j}^{\bar x}(t, x)| \dif x.
$$
Therefore up to extracting a further subsequence $\Phi_{r_j}$ goes to zero pointwisely for a.e. $t$, so that we can assume also
\begin{equation}\label{eq:tconvSc}
    \lim_j \int_{\mathbb R} \phi(y)|\bs u_{\infty}(t,x) - \bs u_{r_j}^{\bar x}(t, x)| \dif x = 0 \quad \text{for all $t \in \mathbb R^+ \setminus S,\quad |S| = 0$.}
\end{equation}
Moreover, since $\phi$ is compactly supported and $\bar x$ is a Lebesgue point of $\bs u(\bar t, \cdot)$ we have
\begin{equation}\label{eq:lebp}
    \lim_j  \int_{\mathbb R} \phi(y) \left|\bs u(\bar t, \bar x + r_j y) -  \bs u(\bar t , \bar x)\right| \dif y  = 0.
\end{equation}
Then we estimate for $t \in S^c$, using \eqref{eq:lebp}, \eqref{eq:tconvSc} in the first two equalities
$$
\begin{aligned}
    \left|\int_{\mathbb R}  \phi(y) \big(\bs u(\bar t, \bar x) -  \bs u^{\bar x}_{\infty}(t,y)\big)\right| \dif y  &  = \lim_{j \to \infty}  \left|\int_{\mathbb R}  \phi(y) \big(\big(\bs u(\bar t, \bar x + r_j y) -  \bs u^{\bar x}_{\infty}(t,y)\big)\right|  \\
    & = \lim_{j \to \infty} \left| \int_{\mathbb R}  \phi(y) \big(\bs u(\bar t, \bar x + r_j y) -  \bs u^{\bar x}_{r_j}(t,y)\big) \dif y\right| \\
    & \leq \int_{\bar t}^t \int_{\mathbb R} |\phi^\prime(y) f(\bs u_{r_j}^{\bar x}(s, y)) |\dif y \dif s \leq  C (t-\bar t)
\end{aligned}
$$
where in the last inequality we used 
$$
\begin{aligned}
    \int_{\mathbb R}  \phi(y) \big(\bs u^{\bar x}_{r_j}(t,y)- \bs u(\bar t, \bar x + r_j y) \big) \dif y & = \int_{\bar t}^t \frac{\dif}{\dif t} \int_{\mathbb R} \phi(y) \bs u_{r_j}^{\bar x}(s, y) \dif y \dif s \\
    & = \int_{\bar t}^t \int_{\mathbb R} \phi^\prime(y) f(\bs u_{r_j}^{\bar x})(s, y) \dif y \dif s
\end{aligned}
$$

Finally, we obtained that $\bs u_\infty^{\bar x}$ is an isentropic solution in $\mathbb R^+\times \mathbb R$ satisfying \eqref{eq:constinit} with $\bs u_0 = \bs u(\bar t, \bar x)$, and therefore by Lemma \ref{lemma:uconst} we obtain  that $\bs u_{\infty}^{\bar x} \equiv \bs u(\bar t, \bar x)$ is constant.
In particular, the limiting constant $\bs u^{\bar x}_{\infty}$ does not depend on the subsequence, therefore we actually obtain the full convergence \eqref{eq:fullconv}.

\vspace{0.3cm}
       \textbf{2.} It follows from \eqref{eq:fullconv} that 
       \begin{equation}\label{eq:point}
       \lim_{r \to 0^+} \int_0^1 \int_{-1}^1 |\bs u(\bar t + s r, \bar x + y r) -\bs u(\bar t, \bar x)| \dif y \dif s = 0 \qquad \text{for a.e. $\bar x \in \mathbb R$}.
       \end{equation}
       Let $\omega : \mathbb R^+ \to \mathbb R^+$ the $\mathbf L^1$ modulus of continuity of $\bs u(\bar t, \cdot)$ on the compact set $[-M -1, M+1]$, i.e.
       \begin{equation}
           \int_{-M-1}^{M+1} |\bs u(\bar t, z) - \bs u(\bar t, z+\delta)|\dif z \leq \omega(|\delta|)
       \end{equation}
       By Fubini's theorem we obtain
\begin{equation}\label{eq:eqfub}
    \begin{aligned}
        & \int_{-M}^M \int_0^1 \int_{-1}^1| \bs u(\bar t + s r, \bar x + y r) - \bs u(\bar t, \bar x)| \dif y \dif s \dif \bar x \\
        & =  \int_{[-1,1]} \int_{-M}^M \int_0^1 |\bs u(\bar t + s r, \bar x + y r) - \bs u(\bar t, \bar x)|\dif \bar x\dif s \dif y  \\
        & = \int_{[-1,1]} \int_0^1 \int_{-M+yr}^{M+yr} |\bs u(\bar t + s r, z) - \bs u(\bar t, z-y r)|\dif z\dif s \dif y \\
        & \leq \int_{[-1,1]} \int_0^1 \int_{-M}^{M} |\bs u(\bar t + s r, z) - \bs u(\bar t, z)|\dif z\dif s \dif y  +   4r\|\bs u\|_{\mathbf L^\infty} +\int_{-1}^1 \omega(|y|r)\dif y \\
        & =2 \fint_{0}^r \int_{-M}^{M} |\bs u(\bar t + s, z) - \bs u(\bar t, z)|\dif z\dif s + o(1)
    \end{aligned}
\end{equation}
where $o(1)$ goes to zero as $r \to 0$.
From \eqref{eq:point}, \eqref{eq:eqfub} and the dominated convergence theorem, we deduce that
$$
\lim_{r \to 0^+} \fint_{0}^r \int_{-M}^{M} |\bs u(\bar t + s, z) - \bs u(\bar t, z)|\dif z\dif s  = \frac{1}{2}  \lim_{r \to 0^+} \int_{-M}^M \int_0^1 \int_{-1}^1|\bs u(\bar t + s r, \bar x + y r) - \bs u(\bar t, \bar x)| \dif y \dif s \dif \bar x  = 0
$$
       \end{proof}

\begin{proof}[Proof of Theorem \ref{thm:strongtime}]
Notice that the limit 
\begin{equation}\label{eq:weaklimu}
\bs u(\bar t, \cdot) = \text{weak$^*$}-\lim_{t \to \bar t} \bs u(t,\cdot)
\end{equation}
exists at every time $\bar t$, as well as the one sided limits
\begin{equation}\label{eq:weaklimeta}
\text{weak$^*$}-\lim_{t \to \bar t^{\pm}} \eta(\bs u(t,\cdot))
\end{equation}
Therefore we only have to prove that the limit in \eqref{eq:weaklimu} is strong, in which case the one sided limits \eqref{eq:weaklimeta} coincide.
If $\bs u(t,x)$ is  a finite entropy solution, $\bs v(t,x) :=\bs u(T-t, -x)$ is again a finite entropy solution in $[0,T]\times \mathbb R$. Therefore, up to this change of variable, to prove Theorem \ref{thm:strongtime} it is enough to prove that $\bs u(t, \cdot) \longrightarrow \bs u(\bar t, \cdot)$ strongly for $t \to \bar t^+$.

    \textbf{1.} Thanks to Lemma \ref{lemma:timestrip}, for every $M > 0$, we deduce that there exists a sequence $t_j \to \bar t^+$ such that 
    $$
    \lim_{j \to +\infty} \int_{-M}^M| \bs u(t_j, x)- \bs u(\bar t,x)| \dif x = 0.
    $$
Let $(\eta, q)$ be any entropy-entropy flux pair.
Computing the limit along the specific sequence $t_j$ we find that
\begin{equation}\label{eq:commutation}
\text{weak$^*$}-\lim_{t \to \bar t^{\pm}} \eta(\bs u(t, \cdot)) = \eta(\text{weak$^*$}-\lim_{t \to \bar t} \bs u(t,\cdot)) = \eta (\bs u(\bar t, \cdot)).
\end{equation}
If a uniformly convex entropy exists, then taking $\eta$ as that entropy the proof can be immediately concluded by a standard argument, using that the weak limit commutates with a uniformly convex function. In the more general case, a bit more work is required, and is the content of the next step.

\vspace{0.3cm}
\textbf{2.} Let now $\{t_j\}$ be a sequence converging to $\bar t$ for which there exists a Young measure $\alpha : \mathbb R \to \mc P(\mc U)$, $x \mapsto \alpha_x \in \mc P(\mc U)$, such that
$$
\bs u(t_j, \cdot) \longrightarrow  \alpha \qquad \text{in the sense of Young measures as $j \to +\infty$}.
$$
We recall that the latter convergence means that for every continuous function $\psi : \mc U \to \mathbb R$ there holds
$$
\psi(\bs u(t_j, \cdot)) \rightharpoonup^\ast f_{\psi}(x) := \int \psi(\bs v) \dif \alpha_{x}(\bs v) \qquad \text{weakly$^\ast$ in $\mathbf L^{\infty}$ as $j \to +\infty$.}
$$
Since every bounded sequence in $\mathbf L^\infty$ admits a subsequence converging in the sense of Young measures, it is clear that to conclude the proof it is enough to show that $\alpha$ is a Dirac Young measure for every such sequence, i.e. that $\alpha_x = \delta_{\bs v(x)}$, for some $\bs v(x) \in \mc U$. 

We first notice that \eqref{eq:commutation}, applied to the sequence $\{t_j\}$, means (taking $\psi = \eta$, where $\eta$ is a smooth entropy):
\begin{equation}\label{eq:comm1}
 \int_{\mc U} \eta(\bs v) \dif \alpha_x(\bs v) = \eta \left( \int_{\mc U} \bs v \dif \alpha_x(\bs v)\right) \qquad \text{for a.e. $x \in \mathbb R$}
\end{equation}
Let $x$ be a point where \eqref{eq:comm1} holds, and by contradiction assume that $\alpha_x$ is not a Dirac mass. Then, without loss of generality, we can assume that $\alpha_x$ is not supported in a single curve $S^1_{\xi} := \{\phi_1(\bs u) = \xi\}$ (otherwise, it is not supported in a single curve of the second Riemann invariant, for which the same analysis can be applied). 
Define the barycenter of $\alpha_x$
$$
\bs b_x := \int_{\mc U} \bs w \dif \alpha_x(\bs w). 
$$
Then at least one of the following holds:
\begin{equation}\label{eq:bsplits}
\alpha_x \big(\{\bs w \; | \; \phi_1(\bs w) < \phi_1(\bs b_x) \}\big) > 0, \qquad  \alpha_x \big(\{\bs w \; | \; \phi_1(\bs w) > \phi_1(\bs b_x) \}\big) > 0.
\end{equation}
Without loss of generality assume that the first one holds.
Let $\rho$ be a smooth function $\rho:[\ubar w, \bar w] \to \mathbb R^+$ and define 
$$
\eta_{\rho}(\bs w) := \int_{\mathbb R} \bs \chi[\xi](\bs w) \rho(\xi) \dif \xi.
$$
Since $\bs \chi[\xi]$ are distributional entropies, it is immediate to check that $\eta_\rho$ is a smooth entropy (see \cite{PT00}, \cite{AMT25}). Therefore by \eqref{eq:comm1} there holds
\begin{equation}
    \int_{\mc U} \int_{\mathbb R} \bs \chi[\xi](\bs w) \rho(\xi) \dif \xi \dif \alpha_x(\bs w) = \int_{\mathbb R} \bs \chi[\xi]\left(\bs b_x\right)  \rho(\xi) \dif \xi 
\end{equation}
But now let $\rho$ be strictly positive in a thin strip 
$$
\{\bs w \in \mc U \;  | \; \phi_1(\bs w) \in (\xi_{\min}-\varepsilon, \xi_{\min}+\varepsilon) \}
$$
and zero outside, where 
$$
\xi_{\min} : = \inf \{\phi_1(\bs w) \;  | \; \bs w \in \mr{supp} \; \alpha_x\}, \qquad \varepsilon < \min\{\bar r/2, \, \mr{dist} (\xi_{\min}, \phi_1(\bs b_x)\}
$$
and $\bar r$ is as in Proposition \ref{prop:localspeed}. Then we have 
$$
0 < \int_{\mc U} \int_{\mathbb R} \bs \chi[\xi](\bs w) \rho(\xi) \dif \xi \dif \alpha_x(\bs w)  = \int_{\mathbb R} \bs \chi[\xi]\left(\bs b_x\right)  \rho(\xi) \dif \xi  = 0
$$
where in the first inequality we used that $\bs \chi[\xi](\bs w) > c > 0$ if 
$$
\xi, \phi_1(\bs w) \in (\xi_{\min}-\varepsilon, \xi_{\min}+\varepsilon), \qquad \xi \leq \phi_1(\bs w)
$$
by Proposition \ref{prop:localspeed}, and in the last inequality we used $\bs \chi[\xi](\bs b_x) = 0$ for all $\xi \in \mr{supp} \, \rho$.
\end{proof}

We conclude the section with a simple corollary of Theorem \ref{thm:strongtime}.

\begin{coro}
    Let $\bs u$ be a finite entropy solution to \eqref{eq:systemi}. Then for every entropy-entropy flux $\eta, q$, the dissipation measure $\mu_\eta$ satisfies
    $$
   |\mu_\eta| \big(\{t = \bar t\} \big) = 0 \qquad \forall \; \bar t > 0.
    $$
\end{coro}
In particular $\bs \nu$ defined by \eqref{eq:supmueta} satisfies $\bs \nu(\{t = \bar t\}) = 0$ for all $\bar t > 0$.
\begin{proof}
Let 
$$
h^\varepsilon(t) := \begin{cases}
    \varepsilon^{-1}(t-\bar t+\varepsilon) & \text{if $\bar t-\varepsilon < t < \bar t$}, \\
    \varepsilon^{-1}(\bar t + \varepsilon-t) & \text{if $\bar t < t < \bar t+ \varepsilon$}, \\
    0 & \text{otherwise}.
\end{cases}
$$
    It is enough to show that for every smooth $\varphi \in C^1_c(\mathbb R)$ there holds
    $$
    \lim_{\varepsilon \to 0^+} \iint \varphi(x) h^\varepsilon(t) \dif \mu_\eta(t,x) = 0.
    $$
By definition of $\mu_\eta$, we have
$$
\begin{aligned}
   \left|\iint \varphi(x) h^\varepsilon(t) \dif \mu_\eta(t,x)\right| & = \left|\iint  \varphi(x) h^\varepsilon_t(t) \eta(\bs u)  + q(\bs u) \varphi^\prime(x) h^\varepsilon(t) \dif x \dif t\right|   \\
    & =\left|\fint_{\bar t}^{\bar t + \varepsilon} \eta(\bs u) \varphi(x) \dif t \dif x - \fint_{\bar t-\varepsilon}^{\bar t}\eta(\bs u) \varphi(x) \dif t \dif x \right| + C \varepsilon.
\end{aligned}
$$
    where $C$ depends on $\varphi$ and $\|\bs u\|_{\mathbf L^\infty}$. The first term goes to zero as $\varepsilon \to 0$ since $\eta$ is Lipschitz and $\bs u(t,\cdot) \to \bs u(\bar t, \cdot)$ strongly in $\mathbf L^1_{loc}(\mathbb R)$, and this concludes the proof.
\end{proof}

\section{Proof of the decay estimate}\label{sec:decay}
\subsection{Idea of the proof} The proof is based on the dispersive properties of the transport equation \eqref{eq:kin22}, \eqref{eq:kin221}. Let us sketch the proof of the decay for the first Riemann invariant $w$. To convey the main idea, assume for the moment that $\bs u$ is a finite entropy solution and that (1) and (2) of Proposition \ref{prop:localspeed} hold for every $\xi \in [\ubar w, \bar w]$ and $(w, z) \in [\ubar w, \bar w] \times [\ubar z, \bar z]$, and that $\mu_0, \mu_1 = 0$ in some time interval $[0,T]$, which means entropy has not been dissipated yet. Let $\bs u$ be an entropy solution whose initial datum $\bs u(0,x)$ is in $\mathbf L^1$ with respect to a constant $\widehat {\bs u} \in \mc U$, i.e. $\bs u(0,x)-\widehat{\bs u} \in \mathbf L^1$. Assume also $w(t,x) > \widehat{w}$, where $\widehat{w} = \phi_1(\widehat{\bs u})$ (the general case would follow by considering the positive and negative part separately).

Here \eqref{eq:kin22} can be rewritten as a transport equation for $\bs \chi_{\bs u}$:
$$
        \; \; \partial_t \bs \chi_{\bs u}(t,x,\xi) + \partial_x  \Big(\bs \lambda_1[\xi](\bs u(t,x))\bs \chi_{\bs u}(t,x,\xi)\Big) = 0 \qquad \text{in $\msc D^\prime$}\big(\Omega \times ({\ubar w}, \bar w) \big)
$$
where $\bs \chi_{\bs u} \equiv \bs \chi_{\bs u} \mathbf 1_{\xi > \widehat{w}}(\xi)$
and where $\bs \lambda_1$ is as in \eqref{eq:lambdadef} where here we have considered $\bs \lambda_1[\xi](\bs u)$ as a function of $\bs u$ or of the Riemann invariants up to the invertible change of variables $(\phi_1, \phi_2) : \mc U \to \mc W$. Now, by our simplifying assumptions, for fixed $\bs u$ we have
$$
\partial_\xi \bs \lambda_1[\xi](\bs u) > c > 0, \qquad \bs \chi_{\bs u}(t,x,\xi) > c > 0.
$$
This means that particles at height $\xi$ will typically cross from the left particles at height $\xi^\prime < \xi$, and once they cross they will never meet again. This basic dispersive mechanism is formalized in the introduction of an interaction functional 
$$
\mc I(t) := \int_{\mathbb R^4} \bs \chi_{u}(t, x, \xi) \bs \chi_{u}(t, x^\prime, \xi^\prime) \mathbf 1_{\xi > \xi^\prime}(\xi, \xi^\prime) \mathbf 1_{x < x^\prime}(x, x^\prime) \dif x \dif x^\prime \dif \xi \dif \xi^\prime.
$$
In this simple case, at least formally, by differentiating we obtain 
$$
\dot{\mc I}(t) := 
\int_{ \mathbb R^3}  \mathbf 1_{\xi > \xi^\prime}(\xi, \xi^\prime) \Big(  \bs \lambda_1[\xi^\prime](\bs u(t,y)) - \bs \lambda_1[\xi](\bs u(t,y)) \Big)\, \bs \chi_{\bs u}(t,y, \xi) \bs \chi_{\bs u}(t,y, \xi^\prime) \dif \xi \dif \xi^\prime \dif y 
$$
and therefore since $$\bs \lambda_1[\xi^\prime](\bs u(t,y)) < -c (\xi-\xi^\prime) + \bs \lambda_1[\xi](\bs u(t,x))$$
we obtain
$$
\dot{\mc I}(t) \leq - C \int_{\mathbb R} |w(t,x)-\widehat w|^3 \dif x
$$
where $\widehat{w} = \phi_1(\widehat{\bs u})$, and by integrating we obtain the desired estimate 
$$
\int_0^T \int_{\mathbb R} |w(t,x)-\widehat w|^3 \dif x \leq \frac{1}{C} \mc I(0).
$$
Since $\mc I(0) < +\infty$, in this case this argument would also yield an upper bound on the time $T$, thus providing an upper bound for the time of shock formation.

In the general case, one has to take into account the fact that (1), (2) of Proposition \ref{prop:localspeed} hold only locally in $\xi \sim \phi_1(\bs u)(t,x)$. This will be achieved by localizing the kinetic equation in strips between two values $\mathbb R^+_t \times \mathbb R_x \times [\ell, \ell + r]$, where $0 < r < \bar r$ is sufficiently small (Lemma \ref{lemma:kinloc}), by proving a decay estimate in each strip (Proposition \ref{prop:locdecayGNL}) and finally concluding by an iterative procedure which allows to pass from one strip to the next one (Theorem \ref{thm:kinit}). Moreover, due to the presence of the term $\partial_\xi \mu_1$, one has to add a weight of the form $(\xi-\xi^\prime)$ into the definition of the functional, which in turn allows to control only the $\mathbf L^4$ norm of the solution.

For vanishing viscosity solutions, additional structure on the dissipation measures can be obtain, see \cite{AMT25}.
\begin{thm}[\cite{AMT25}, Theorem 1.1] \label{thm:kin}
Let $\bs u: \Omega \to \mc U$ be a vanishing viscosity solution to \eqref{eq:systemi}, assume that \eqref{eq:systemi} admits a uniformly convex entropy $E : \mc U \to \mathbb R$. Define $\bs \chi_{\bs u}, \bs \psi_{\bs u}$ and $\bs \upsilon_{\bs u}, \bs \varphi_{\bs u}$ as in \eqref{eq:kfdef}, \eqref{eq:kfdef1}. Then there are locally finite measure $\mu_0,\mu_1 \in \msc M(\Omega \times ({\ubar w}, \bar w))$ and $\nu_0,\nu_1 \in \msc M(\Omega \times ({\ubar z}, \bar z))$ such that
\begin{enumerate}
\item Equations \eqref{eq:kin22}, \eqref{eq:kin221} hold with $\mu_1, \nu_1 \geq 0$.
    \item  For some constant $C > 0$, we have
         $$
 ({\mathtt p_{t,x}})_\sharp |\mu_0| +  ({\mathtt p_{t,x}})_\sharp |\nu_0| \leq C \, ({\mathtt p_{t,x}})_\sharp \mu_1 + ({\mathtt p_{t,x}})_\sharp \nu_1. 
 $$
\item For every $M, T > 0$, there holds 
    \begin{equation}
\int_0^T \int_{-M-L(T-t)}^{M + L(T-t)} \int_{\mathbb R} \dif |\mu_i|  \leq C \int_{-M-LT}^{M + LT} E(\bs u(0, x)) \dif x
    \end{equation}
    \begin{equation}
      \int_0^T \int_{-M-L(T-t)}^{M + L(T-t)}\int_{\mathbb R} \dif |\nu_i| \leq C \int_{-M-LT}^{M + LT} E(\bs u(0, x)) \dif x.
 \end{equation}
 where $C, L > 0$ depends only on $E$ and on $f$.
 \end{enumerate}
\end{thm}
Here $\mathtt p_{t,x}$ denotes the canonical projection on the $t,x$ variables. We recall that given a measurable map $f: X \to Y$ between measure spaces $X, Y$, for any $\mu \in \msc M(X)$ the pushforward measure $f_\sharp \mu \in \msc M(Y)$ is defined by 
$$
f_\sharp \mu(A) = \mu(f^{-1}(A)) \qquad \forall \; \text{measurable} \; A \subset Y.
$$

In the following $\bs u(t,x)$ is a vanishing viscosity solution such that $\bs u(0, x) -\widehat{\bs u} \in \mathbf L^1$, for some $\widehat{\bs u} \in \mc U$, and $\widehat{w} = \phi_1(\widehat{\bs u})\in (\ubar w, \bar w)$. In the following, we will focus on estimating $[w(t,x)-\widehat{w}]^+$, the estimate for the negative part being entirely symmetric.

 For $r > 0$ and $\ell \in [\widehat{w}, \bar w-r]$ we define 
\begin{equation}\label{eq:chirldef}
\bs \chi^{r, \ell}(t,x,\xi) := \bs \chi_{\bs u}(t,x,\xi) \mathbf 1_{\{\ell < \xi < \ell + r\}}(t,x,\xi)
\end{equation}
\begin{equation}
    \bs \psi^{r, \ell}(t,x,\xi) := \bs \psi_{\bs u}(t,x,\xi) \mathbf 1_{\{\ell < \xi < \ell + r\}}(t,x,\xi).
\end{equation}

\begin{lemma}\label{lemma:kinloc}
    In the notation defined above, there are positive, locally finite measures $ \ms f^{\ell+ r}_{\mr{in}}, \ms f^{\ell}_{\mr{out}}$ such that:
    \begin{enumerate}
        \item[$(i)$] The following equation holds in the sense of distributions:
    \begin{equation}\label{eq:kinloc}
    \begin{aligned}
    \partial_t \bs \chi^{r,\ell}  + \partial_x  \bs \psi^{r, \ell}  & = \partial_{\xi} [ \mu_1 \cdot \mathbf 1_{\{\ell < \xi < \ell+r\}}] \\
    & +  \mu_0\cdot \mathbf 1_{\{\ell < \xi <\ell+r\}} + \ms f^{\ell+ r}_{\mr{in}} - \ms f^{\ell}_{\mr{out}}  \qquad \text{in $\msc D^\prime_{t,x,\xi}$}.
        \end{aligned}
        \end{equation}
\item[$(ii)$] There is $C >0$, depending only on the system, such that 
        \begin{equation}
           \iint_{\mathbb R^+ \times \mathbb R^2} \dif\,  \ms f^{\ell+ r}_{\mr{in}}(t,x, \xi) \leq C \cdot \Vert{\bs u_0} - \widehat{\bs u}\Vert_{\mathbf L^1}, \quad  \iint_{\mathbb R^+ \times \mathbb R^2} \dif\,  \ms f^{\ell}_{\mr{out}}(t,x, \xi) \leq C \cdot \Vert{\bs u_0} - \widehat{\bs u}\Vert_{\mathbf L^1}
        \end{equation}
        \item[$(iii)$] For every time $t> 0$ it holds:
\begin{equation}\label{eq:(iii)}
    \int_{\mathbb R^2} |\bs \chi^{r, \ell}(t,x,\xi)| \dif x \dif \xi < C \|\bs u_0-\widehat{\bs u}\|_{\mathbf L^1}
\end{equation}
        \end{enumerate}

\end{lemma}
\begin{proof}

We first prove $(i)$. For $0 < \eps < r/2$, define an $\eps^{-1}$-Lipschitz function $h_\eps$ by
    $$
    h_\eps(\xi)\doteq \begin{cases}
        0 & \text{if $\xi \leq \ell$} \\
        (\xi-\ell)\eps^{-1} & \text{ if $\ell \leq \xi \leq \ell +\eps$}\\
         1 & \text{if $\ell +\eps \leq \xi \leq \ell+r-\eps$} \\
         (\ell+r-\xi)\eps^{-1} & \text{if $\ell+r-\eps \leq \xi \leq \ell+r$} \\
         0 & \text{if $\xi \geq \ell + r$}
    \end{cases}
    $$
    and notice that 
    $$
    \bs \chi_{\bs u} h_\eps \; \longrightarrow \; \bs \chi^{r, \ell} \qquad \text{in $\mathbf L^1_{loc}$ as $\eps \to 0^+$}.
    $$
For any test function $\varphi(t,x,\xi)$ we compute
\begin{equation}\label{eq:chi+comp}
\begin{aligned}
    \int_{\mathbb R^+ \times \mathbb R^2} & \big[ \varphi_t \bs \chi^{r, \ell} + \varphi_x\bs \psi^{r, \ell}\big]\dif \xi \dif x \dif t \\
    & = \lim_{\eps \to 0^+} \int_{\mathbb R^+ \times \mathbb R^2} \big[ \varphi_t h_\eps\bs \chi_{\bs u} + \varphi_x h_\eps\bs \psi_{\bs u}\big]\dif \xi \dif x \dif t \\
    & =  \lim_{\eps \to 0^+} \int_{\mathbb R^+ \times \mathbb R^2}  (\varphi  h^\prime_\eps + \varphi_\xi h_\eps) \dif \mu_1  -  \int_{\mathbb R^+ \times \mathbb R^2} \varphi h_\eps \dif \mu_0.
\end{aligned}
\end{equation}
We now compute
\begin{equation}\label{eq:49}
\begin{aligned}
    -\int_{\mathbb R^+ \times \mathbb R^2} \varphi h_\eps^\prime \dif \mu_1 & =  \frac{1}{\eps}\int_{\mathbb R^+ \times \mathbb R^2} \varphi \mathbf 1_{\{\ell \leq \xi \leq \ell +\eps\}}(\xi) \dif \mu_1 \\
    & -\frac{1}{\eps}\int_{\mathbb R^+ \times \mathbb R^2} \varphi \mathbf 1_{\{\ell+r-\eps \leq \xi \leq \ell+r \}}(\xi) \dif \mu_1
\end{aligned}
\end{equation}
Next we compute the limits
in the right hand side, we start with the first term. Define a test function 
$$
g_\eps(\xi) = \begin{cases}
     0 & \text{if $\xi \leq \ell$} \\
        (\xi-\ell)\eps^{-1} & \text{ if $\ell \leq \xi \leq \ell +\eps$}\\
         1 & \text{if $\ell +\eps \leq \xi $}.
\end{cases}
$$
Now, for every $\varphi \in C^1_c(\mathbb R^+ \times \mathbb R^2)$ we test $\varphi g_\eps$ against \eqref{eq:kin22}, so that we obtain using \eqref{eq:49}
$$
\begin{aligned}
\int_{\mathbb R^+ \times \mathbb R^2} \varphi g_\eps^\prime \dif \mu_1  & =\int_{\mathbb R^+ \times \mathbb R^2} \varphi g_\eps \dif \mu_0 + \int_{\mathbb R^+ \times \mathbb R^2} \big( \varphi_t \bs \chi_{\bs u} + \varphi_x \bs \psi_{\bs u}\big) g_\eps \dif x \dif t \dif \xi  - \int_{\mathbb R^+ \times \mathbb R^2} \varphi_{\xi} g_\eps \dif \mu_1.
\end{aligned}
$$
Now taking the limit we obtain 
\begin{equation}\label{eq:limmu1eps}
\begin{aligned}
   \lim_{\eps \to 0^+} & \int_{\mathbb R^+ \times \mathbb R^2} \varphi  \dif \Big( \frac{1}{\eps}\mu_1\llcorner \, \big(\mathbb R^+ \times \mathbb R \times (\ell, \ell + \eps)  \big)\Big) \\
    = & \int_{\mathbb R^+ \times \mathbb R^2} \varphi \mathbf 1_{\{\xi > \ell\}}(\xi) \dif \mu_0 - \int_{\mathbb R^+ \times \mathbb R^2} \varphi_\xi \mathbf 1_{\{\xi > \ell\}} \dif \mu_1 \\
    + & \int_{\mathbb R^+ \times \mathbb R^2}  \big( \varphi_t \bs \chi[\xi](\bs u) + \varphi_x \bs \psi[\xi](\bs u)\big)\mathbf 1_{\xi > \ell}(\xi) \dif x \dif t \dif \xi.
\end{aligned}
\end{equation}
Therefore we deduce that the limit 
$$
\ms f^{\ell}_{\mr{out}} \doteq \lim_{\eps \to 0^+} \frac{1}{\eps}\mu_1\llcorner \, \big(\mathbb R^+ \times \mathbb R \times (\ell, \ell + \eps)  \big)
$$
exists in the sense of distributions; moreover, since $\mu_1$ it is a positive measure, the limit must be a positive distribution and therefore is a locally finite positive measure that we call $\ms f^{\ell}_{\mr{out}} \in \msc M(\mathbb R^+\times \mathbb R^2)$. Doing the same for the second term and inserting the result in \eqref{eq:chi+comp} we deduce
\begin{equation}\label{eq:chi+comp1}
\begin{aligned}
    \int_{\mathbb R^+ \times \mathbb R^2} \big[ \varphi_t \bs \chi^{r, \ell}&  + \varphi_x\bs \psi^{r, \ell}\big]\dif \xi \dif x \dif t  = \int_{\mathbb R^+ \times \mathbb R^2} \varphi \dif ({\ms f^\ell_{\ms{out}} - \ms f^{\ell +r}_{\ms{in}}}) \\
    & + \int_{\mathbb R^+\times \mathbb R^2} \varphi_{\xi} \mathbf 1_{\{\ell < \xi < \ell + r\}} \dif \mu_1 -  \int_{\mathbb R^+ \times \mathbb R^2} \varphi \mathbf 1_{\{\ell < \xi < \ell + r\}} \dif \mu_0. 
\end{aligned}
\end{equation}

Next, we turn to the proof of $(ii)$.  First, we notice that, since $E$ is uniformly convex and smooth, up to summing an affine function to $E$, we can assume that $E \geq 0$, $E(\widehat{\bs u}) = 0$, and that for some $c > 0$
\begin{equation}\label{eq:Esquare}
\frac{1}{c}|\bs u- \widehat{\bs u}|^2 \leq E(\bs u) \leq c |\bs u- \widehat{\bs u}|^2 \qquad \text{$\forall \, \bs u \in \mc U$}.
\end{equation}
Therefore we have
$$
\int_{\mathbb R} E(\bs u_0(x)) \dif x \leq c \int_{\mathbb R} |\bs u_0(x) - \widehat{\bs u}|  \dif x.
$$
This implies that letting $M \to +\infty$ in point 3 of Theorem \ref{thm:kin}, we obtain 
\begin{equation}\label{eq:muiglobal}
|\mu_i|((0,T) \times \mathbb R) \leq C \| \bs u_0-\widehat{\bs u}\|_{\mathbf L^1(\mathbb R)}< +\infty.
\end{equation}
By (1) of Proposition \ref{prop:localspeed}, there holds
$$
\bs \chi[\xi](\bs u(t,x)) <0 \quad \Longrightarrow \quad \xi < \phi_1(\bs u(t,x))- \bar r
$$
and therefore
\begin{equation}\label{eq:chiL1}
\begin{aligned}
     \int_{\mathbb R^+\times \mathbb R} & \mathbf 1_{\xi \geq \ell}(\xi) [\bs \chi[\xi]({\bs u}(t,x))]^- \dif x\dif \xi \\
     & \leq (\bar w- \ell) \sup_{t,x, \xi} |\bs \chi| \mathscr L^1\Big(\big\{x \in \mathbb R\; | \; \phi_1(\bs u(t,x)) > \ell + \bar{r}  \big\}\Big)
\end{aligned}
\end{equation}
where we used that $\bar w \geq \xi \geq \ell$ in the integrand.
Here if $a \in \mathbb R$:
$$
[a]^- := \max\{0, -a\}, \qquad [a]^+ := \max\{0, a\}.
$$
Moreover, by the Markov's inequality, using $\phi_1(\widehat{\bs u}) < \ell$,
\begin{equation}\label{eq:chiL2}
\begin{aligned}
    \bar r^2 \mathscr L^1\Big(\big\{x \in \mathbb R\; | \; \phi_1(\bs u(t,x)) > \ell + \bar{r}  \big\}\Big) 
    & \leq   \int_{\mathbb R} |\phi_1(\bs u(t,x))- \ell|^2 \dif x \\
    & \leq \int_{\mathbb R} |\phi_1(\bs u(t,x))- \phi_1(\widehat{\bs u})|^2 \dif x \\
    & \leq \mr{Lip}(\phi_1) \int_{\mathbb R} |\bs u(t,x)-\widehat{\bs u}|^2 \dif x\\
    & \leq c \,\mr{Lip}(\phi_1)  \int_{\mathbb R} E(\bs u(t,x))\dif x \\
    & \leq c \,\mr{Lip}(\phi_1)   \int_{\mathbb R} E(\bs u_0(x))\dif x  \\
    & \leq c^2 \,\mr{Lip}(\phi_1)  \| \bs u_0-\widehat{\bs u}\|_{\mathbf L^2}^2.
     \end{aligned}
\end{equation}
Therefore combining \eqref{eq:chiL1}, \eqref{eq:chiL2}, we find that there is a positive constant $C$ such that
\begin{equation}\label{eq:negpart}
    \begin{aligned}
     \int_{\mathbb R^+\times \mathbb R} & \mathbf 1_{\xi \geq \ell}(\xi) [\bs \chi[\xi]({\bs u}(t,x))]^- \dif x\dif \xi \leq C \| \bs u_0-\widehat{\bs u}\|_{\mathbf L^1}
\end{aligned}
\end{equation}

Then we choose 
\begin{equation}\label{eq:varphi}
\varphi(t,x) = \mathbf 1_{(0,T)\times \mathbb R}(t,x) \qquad \forall \; t, x \in \mathbb R^+ \times \mathbb R
\end{equation}
and we compute, using \eqref{eq:limmu1eps} and a standard mollification argument on $\varphi$,
$$
\begin{aligned}
\int_{\mathbb R^+ \times \mathbb R^2} \dif \ms f^{\ell, r}_{\ms {out}}(x, t, \xi) &  \leq \lim_{T\to \infty} |\mu_0|((0,T) \times \mathbb R^2))  -\lim_{T \to \infty} \int_{\mathbb R^+\times \mathbb R} \mathbf 1_{\xi \geq \ell}(\xi) \bs \chi[\xi]({\bs u}(T,x)) \dif x\dif \xi \\ 
& + \int_{\mathbb R^+\times \mathbb R} \mathbf 1_{\xi \geq \ell}(\xi) \bs \chi[\xi](\bs u_0) \dif x\dif \xi \leq C \|\bs u_0 - \widehat{\bs u}\|_{\mathbf L^1(\mathbb R)}
\end{aligned}
$$
where $C >0$ is another constant, and we used \eqref{eq:muiglobal} to estimate the first term,  \eqref{eq:negpart} to estimate the second term, and finally the fact that the third integrand is not zero only for $\xi$ such that $\ell \leq \xi \leq w_0(x)$ with $\phi_1(\widehat{\bs u}) \leq \ell$ implies that 
\begin{equation}\label{eq:initdatachi}
    \int_{\mathbb R^+\times \mathbb R} \mathbf 1_{\xi \geq \ell}(\xi) \bs \chi[\xi](\bs u_0) \dif x\dif \xi \leq \sup \, \bs \chi \|\bs u_0 - \widehat{\bs u}\|_{\mathbf L^1(\mathbb R)}
\end{equation}

An entirely similar argument proves that 
$$
 \ms f^{\ell+ r}_{\mr{in}} \doteq \lim_{\eps \to 0^+} \frac{1}{\eps}\mathbf 1_{\{\ell+r-\eps \leq \xi \leq \ell+r \}}(\xi) \cdot  \mu_1.
$$
and
$$
\int_{\mathbb R^+ \times \mathbb R^2} \dif \ms f^{\ell, r}_{\ms {in}}(x, t, \xi) \leq C \cdot \Vert{\bs u_0}-\widehat{\bs u}\Vert_{\mathbf L^1}
$$

To prove (iii), we again test \eqref{eq:chi+comp1} with $\varphi$ as in \eqref{eq:varphi}, and we obtain 
$$
\begin{aligned}
\int_{\mathbb R^2}& \mathbf 1_{\xi \geq \ell}(\xi) [\bs \chi[\xi](\bs u(T,x))]^+ \dif x \dif \xi \leq - \int_{\mathbb R^2} \mathbf 1_{\xi \geq \ell}(\xi) [\bs \chi[\xi](\bs u(T,x))]^- \dif x \dif \xi \\
& +\int_{\mathbb R^2} \mathbf 1_{\xi \geq \ell}(\xi) \bs \chi[\xi](\bs u_0(x)) \dif x \dif \xi + \int_0^T \int_{\mathbb R^2} \dif \big(\ms f_{\ms{in}}^{\ell +r} +\ms f_{\ms{out}}^\ell + |\mu_0| \big)
\end{aligned}
$$
Using \eqref{eq:negpart}, \eqref{eq:initdatachi},  \eqref{eq:muiglobal} and point (ii) to estimate all the terms in the right hand side, we finally obtain \eqref{eq:(iii)}.
\end{proof}

We introduce the following function
\begin{equation}\label{eq:gdef}
g^{r,\ell}(t,x) \doteq \begin{cases}
    w(t,x)-\ell & \text{if $\ell \leq w(t,x) \leq \ell + r$}, \\
    0 & \text{if $w(t,x) < \ell$}, \\
    r & \text{if $w(t,x) > \ell+ r$}.
\end{cases}
\end{equation}

In the following proposition we prove a localized version of the decay estimate, for the function $g^{r, \ell}$ in \eqref{eq:gdef}.

\begin{prop}\label{prop:locdecayGNL}
Assume that the eigenvalue $\lambda_1$ is genuinely nonlinear and let $g^{r, \ell}$ be defined as in \eqref{eq:gdef}, with $r \leq \bar r$. Then there is $C >0$ such that
    \begin{equation}
    \begin{aligned}
    \int_0^{+\infty}\int_\mathbb R\big( g^{r,\ell}(t,x)\big)^4 \dif x \dif t & \leq C \cdot \Big(\Vert{\bs u_0} -\widehat{\bs u}\Vert_{\mathbf L^1} + \Vert{\bs u_0} -\widehat{\bs u}\Vert_{\mathbf L^1}^2\Big)\\
    & + C \cdot \msc L^2\Big( \big\{(t,x) \in \mathbb R^+ \times \mathbb R \; | \; w(t,x) \geq  \ell +r \big\} \Big).
        \end{aligned}
    \end{equation}
\end{prop}

\begin{proof}
We introduce the following interaction functional
\begin{equation}
\mc Q(t) :=  \int_{\mathbb R^4}  \mathbf 1_{\xi^\prime \leq \xi}(\xi, \xi^\prime) \mathbf 1_{x \leq x^\prime}(x, x^\prime) (\xi-\xi^\prime)\, \bs \chi^{r, \ell}(t,x, \xi) \bs \chi^{r, \ell}(t,x^\prime, \xi^\prime) \dif x \dif \xi \dif x^\prime \dif \xi^\prime. 
\end{equation}
 Notice that $\mc Q(t)$ is finite for all $t > 0$ because
 \begin{equation}\label{eq:Qbound}
 \mc Q(t) \leq \bar r\Big( \int_{\mathbb R^2} |\bs \chi^{r, \ell}(t,x, \xi)| \dif x \dif \xi \Big)^2 \leq C \|\bs u_0- \widehat{\bs u}\|_{\mathbf L^1}^2
  \end{equation}
by (iii) of Lemma \ref{lemma:kinloc}.
Let $\varrho : \mathbb R^3 \to \mathbb R$ be smooth, compactly supported, with 
$$
\int_{\mathbb R^3} \varrho = 1
$$
and define 
$$
\varrho^\delta(\cdot) := \frac{1}{\delta^3}\varrho(\frac{\cdot}{\delta}).
$$
Define
$$
\bs \chi^{r,\ell}_\delta  \doteq \bs \chi^{r, \ell} \ast \varrho_\delta, \quad \bs \psi^{r,\ell}_\delta  \doteq \bs \psi^{r, \ell} \ast \varrho_\delta$$
$$
\mu_i^\delta \doteq (\mu_i\mathbf 1_{\{\ell < \xi < \ell + r\}}(\xi)) \ast \varrho_\delta \quad \text{for $i = 0,1$}
$$
and
$$
\ms f_{\ms{in}}^\delta = \ms f^{r, \ell}_{\ms{in}} \ast \varrho_\delta, \qquad \ms f_{\ms{out}}^\delta = \ms f^{r, \ell}_{\ms{out}} \ast \varrho_\delta
$$
Therefore, taking the convolution of \eqref{eq:kinloc} with $\varrho_\delta$, we obtain 
$$
\partial_t \bs \chi^{r,\ell}_\delta + \partial_x \bs \psi^{r,\ell}_\delta = \partial_\xi \mu^{\delta}_1 + \mu^{\delta}_0 + \ms f_{\ms{in}}^\delta  -\ms f_{\ms{out}}^\delta \qquad \forall t > 0, \quad x, \xi \in \mathbb R
$$
Define also 
\begin{equation}\label{eq:func1}
 \mc Q^\delta(t) \doteq  \int_{\mathbb R^4} \mathbf 1_{\xi^\prime \leq \xi}(\xi, \xi^\prime) \mathbf 1_{x \leq x^\prime}(x, x^\prime) (\xi-\xi^\prime)\, \bs \chi_\delta^{r, \ell}(t,x, \xi) \bs \chi_\delta^{r, \ell}(t,x^\prime, \xi^\prime) \dif x \dif \xi \dif x^\prime \dif \xi^\prime. 
  \end{equation}
Since 
$$
\bs \chi_\delta^{r, \ell}(t,x, \xi) \bs \chi_\delta^{r, \ell}(t,x^\prime, \xi^\prime)  \longrightarrow \bs \chi^{r, \ell}(t,x, \xi) \bs \chi^{r, \ell}(t,x^\prime, \xi^\prime) \quad \text{strongly in $\mathbf L^1_{loc}(\mathbb R^+ \times \mathbb R^4)$}
$$
we have for all $T > 0$
$$
\lim_{\delta \to 0} \int_0^T |\mc Q^\delta(t) - \mc Q(t)| \dif t  = 0 
$$
from which we deduce that, up to a subsequence $\{\delta_j\}$, we have
\begin{equation}\label{eq:Qdeltaconv}
\mc Q^{\delta_j}(t) \longrightarrow \mc Q(t) \qquad \text{for a.e. $t > 0$}.
\end{equation}
Compute
 \begin{equation}\label{eq:Qdeltadif}
\begin{aligned}
    \frac{\dif}{\dif t}  \mc Q_\phi^\delta(t)   \doteq & \frac{\dif}{\dif t}\int_{\mathbb R^4} \mathbf 1_{\xi^\prime \leq \xi}(\xi, \xi^\prime) \mathbf 1_{x \leq x^\prime}(x, x^\prime) (\xi-\xi^\prime)\, \bs \chi^{r,\ell}_\delta(t,x, \xi) \bs \chi^{r,\ell}_\delta(t,x^\prime, \xi^\prime) \dif X \\
     =& -\int_{\mathbb R^3}  \mathbf 1_{\xi^\prime \leq \xi}(\xi, \xi^\prime) (\xi-\xi^\prime) \bs \psi^{r,\ell}_\delta(t,x,\xi) \bs \chi^{r,\ell}_\delta(t,x, \xi^\prime)\dif x \dif \xi \dif \xi^\prime \\
    &+  \int_{x \leq x^\prime, \, \xi^\prime \leq \xi} (\xi-\xi^\prime) \left(\partial_\xi \mu_1^\delta(t,x,\xi)+\mu_0^\delta(t,x,\xi)+(\ms f^\delta_{\ms{in}} -\ms f^{\ell, \delta}_{\mr{out}})(t,x,\xi)\right) \bs \chi^{r,\ell}_\delta(t,x^\prime, \xi^\prime) \dif X \\
    & +   \int_{\mathbb R^3} \mathbf 1_{\xi^\prime \leq \xi}(\xi, \xi^\prime)  (\xi-\xi^\prime)\, \bs \psi^{r,\ell}_\delta (t,x,\xi^\prime) \bs \chi^{r,\ell}_\delta(t,x, \xi) \dif x \dif \xi \dif \xi^\prime \\
      &+  \int_{x \leq x^\prime, \, \xi^\prime \leq \xi}   (\xi-\xi^\prime) \left(\partial_{\xi^\prime} \mu_1^\delta(t,x^\prime,\xi^\prime)+\mu_0^\delta(t,x^\prime,\xi^\prime)+(\ms f^\delta_{\ms{in}} -\ms f^{\ell, \delta}_{\mr{out}})(t,x^\prime,\xi^\prime)\right) \bs \chi^{r,\ell}_\delta(t,x, \xi) \dif X \\
     =: \; & I_1(t) + I_2(t) + I_1^\prime(t) + I_2^\prime(t).
\end{aligned}
 \end{equation}
where $\dif X =  \dif x \dif \xi \dif x^\prime \dif \xi^\prime$ denotes integration against the set of variables $(x, \xi, x^\prime, \xi^\prime)$.
We first notice that $I_2(t)$ can be estimated by 
\begin{equation}\label{eq:I2est}
\begin{aligned}
    |I_2(t)|  \leq C \cdot \int_{\mathbb R^2} & \Big(\mu_1^\delta(t,x,\xi) + |\mu^\delta_0(t,x,\xi)| \\
    & + \ms f_{\ms{in}}^\delta(t,x,\xi) + \ms f_{\ms{out}}^\delta(t,x,\xi)\Big) \dif x \dif \xi \int_{\mathbb R^2} |\bs \chi_\delta^{r, \ell}(t,x, \xi)| \dif x \dif \xi\\
    & \leq C \cdot \int_{\mathbb R^2}  \Big(\mu_1^\delta(t,x,\xi) + |\mu_0(t,x,\xi)| \\
     & + \ms f_{\ms{in}}^\delta(t,x,\xi) + \ms f_{\ms{out}}^\delta(t,x,\xi)\Big) \dif x \dif \xi \|\bs u_0-\widehat{\bs u}\|_{\mathbf L^1}
    \end{aligned}
\end{equation}
where we used Lemma \ref{lemma:kinloc} to obtain the last inequality.
Integrating in time and using again Lemma \ref{lemma:kinloc} and \eqref{eq:muiglobal} we obtain for all $T > 0$
\begin{equation}\label{eq:I2est}
    \begin{aligned}
        \int_0^{T} |I_2(t)| \dif t \leq  C  \cdot  \|\bs u_0-\widehat{\bs u}\|_{\mathbf L^1}^2
    \end{aligned}
\end{equation}
where $C$ is a fixed constant depending on the system.
 A completely similar inequality can be obtain for $I^\prime_2(t)$:
\begin{equation}\label{eq:I2est1}
\int_0^{T} |I_2^\prime(t)|\dif t \leq C  \cdot \| \bs u_0- \widehat{\bs u}\|^2_{\mathbf L^1}
\end{equation}
Therefore integrating in time \eqref{eq:Qdeltadif} and usign \eqref{eq:I2est}, \eqref{eq:I2est1}, we get for $T_0 < T$
\begin{equation}\label{eq:Qdeltaineq}
    \begin{aligned}
         \mc Q^\delta(T) -\mc Q^\delta(T_0)  &  \leq \int_{T_0}^T \int_{\mathbb R^3}\mathbf 1_{\{\xi^\prime \leq \xi\}} (\xi-\xi^\prime) \big[ \bs \psi^{r, \ell}_\delta (t, y, \xi^\prime) \bs \chi^{r, \ell}_\delta(t,y, \xi) \\
         & \quad- \bs \psi^{r, \ell}_\delta (t, y, \xi) \bs \chi^{r, \ell}_\delta(t,y, \xi^\prime)\big] \dif \xi \dif \xi^\prime \dif y \dif t 
 + C\cdot  \Vert{\bs u_0}-\widehat{\bs u}\Vert_{\mathbf L^1}.
    \end{aligned}
\end{equation}
Clearly
$$
\Phi_\delta(t,y,\xi, \xi^\prime) := \bs \psi^{r, \ell}_\delta (t, y, \xi^\prime) \bs \chi^{r, \ell}_\delta(t,y, \xi) - \bs \psi^{r, \ell}_\delta (t, y, \xi) \bs \chi^{r, \ell}_\delta(t,y, \xi^\prime)
$$
satisfies
$$
\Phi_\delta \longrightarrow \bs \psi^{r, \ell} (t, y, \xi^\prime) \bs \chi^{r, \ell}(t,y, \xi) - \bs \psi^{r, \ell} (t, y, \xi) \bs \chi^{r, \ell}(t,y, \xi^\prime) \quad \text{strongly in $\mathbf L^1_{loc}(\mathbb R^+ \times \mathbb R^3)$}.
$$
Therefore we can take the limit $\delta \to 0^+$ in \eqref{eq:Qdeltaineq}, using also \eqref{eq:Qdeltaconv} in the left hand side, and we obtain for a.e. $T_0 < T$
\begin{equation}
    \begin{aligned}
         \mc Q(T) -\mc Q(T_0)  &  \leq \int_{T_0}^T \int_{\mathbb R^3}\mathbf 1_{\{\xi^\prime \leq \xi\}} (\xi-\xi^\prime)^+ \big[ \bs \psi^{r, \ell} (t, y, \xi^\prime) \bs \chi^{r, \ell}(t,y, \xi) \\
         & \quad- \bs \psi^{r, \ell} (t, y, \xi) \bs \chi^{r, \ell}(t,y, \xi^\prime)\big] \dif \xi \dif \xi^\prime \dif y \dif t 
 + C\cdot  \Vert{\bs u_0}-\widehat{\bs u}\Vert_{\mathbf L^1}.
    \end{aligned}
\end{equation}

Using Proposition \ref{prop:localspeed}, we further estimate the right hand side above and obtain 
$$
\begin{aligned}
     \mc Q(T) -\mc Q(T_0) & \leq \int_{T_0}^T \int_{ \mathbb R^3} (\xi-\xi^\prime)^+ \big[  \bs \lambda_1[\xi^\prime](\bs u(t,y))\\
     & - \bs \lambda_1[\xi](\bs u(t,y)) \big]\cdot \bs \chi^{r, \ell}(t,y, \xi) \bs \chi^{r, \ell}(t,y, \xi^\prime)\cdot \mathbf 1_{\{w(t,y) \leq \ell +r\}} \dif \xi \dif \xi^\prime \dif y \dif t \\
         & +C \cdot \msc L^2\Big( \big\{(t,x) \in \mathbb R^+ \times \mathbb R \; | \; w(t,x) \geq  \ell +r \big\} \Big)\\
         & + C \cdot \Vert{\bs u_0}-\widehat  {\bs u}\Vert_{\mathbf L^1}.
\end{aligned}
$$
where we used the fact that if $\ell \leq w(t,y) \leq \ell + r$, then for $\xi \in [\ell, w(t,y)]$ we also have $\xi \in (w(t,y)-r, w(t,y))$, so that Proposition \ref{prop:localspeed} applies and we get
$$
\begin{aligned}
\bs \psi^{r, \ell} (t, y, \xi^\prime) & \bs \chi^{r, \ell}(t,y, \xi)  - \bs \psi^{r, \ell} (t, y, \xi) \bs \chi^{r, \ell}(t,y, \xi^\prime) \\
& = \big[  \bs \lambda_1[\xi^\prime](\bs u(t,y))- \bs \lambda_1[\xi](\bs u(t,y)) \big] \bs \chi^{r, \ell}(t,y, \xi) \bs \chi^{r, \ell}(t,y, \xi^\prime).
\end{aligned}
$$
Now again Proposition \ref{prop:localspeed} yields  $\bs \chi^{r, \ell}(t,y, \xi) > c> 0$ for $\xi \in (w(t,y)-r, w(t,y))$, so that
$$
\begin{aligned}
\big[  \bs \lambda_1[\xi^\prime](\bs u(t,y))- \bs \lambda_1[\xi](\bs u(t,y)) \big] \bs \chi^{r, \ell}(t,y, \xi) \bs \chi^{r, \ell}(t,y, \xi^\prime) \leq -c (\xi-\xi^\prime)\bs \chi^{r, \ell}(t,y, \xi) \bs \chi^{r, \ell}(t,y, \xi^\prime).
\end{aligned}
$$
therefore, for some positive constant $C > 0$, we obtain

\begin{equation}
    \begin{aligned}
        \int_{ \mathbb R^3} (\xi-\xi^\prime)^+ & \big[  \bs \lambda_1[\xi^\prime](\bs u(t,y))
      - \bs \lambda_1[\xi](\bs u(t,y)) \big]\cdot \bs \chi^{r, \ell}(t,y, \xi) \bs \chi^{r, \ell}(t,y, \xi^\prime)\cdot \mathbf 1_{\{w(t,y) \leq \ell + r\}}(y) \dif \xi \dif \xi^\prime \dif y \\
      & \leq -c \int_{\mathbb R^3}[(\xi-\xi^\prime)^+]^2\bs \chi^{r, \ell}(t,y, \xi) \bs \chi^{r, \ell}(t,y, \xi^\prime) \mathbf 1_{\{w(t,y) \leq \ell + r\}}(y)\dif y \dif \xi \dif \xi^\prime \\
      & \leq -c^3 \int_{\mathbb R} \int_\ell^{w(t,y)} \int_\ell^{w(t,y)}[(\xi-\xi^\prime)^+]^2 \dif \xi \dif \xi^\prime\mathbf 1_{\{w(t,y) \in (\ell, \ell + r)\}}(y) \dif y   \\
      & \leq -\wt c\int_{\mathbb R} (w(t,y)-\ell)^4 \mathbf 1_{\{w(t,y) \in (\ell, \ell + r)\}}(y) \dif y \\
      & \leq  -\wt c\int_{\mathbb R} (g^{r,\ell}(t,y))^4 \mathbf 1_{\{w(t,y) \in (\ell, \ell + r)\}}(y) \dif y + \wt C \msc L^2\Big( \big\{(t,x) \in \mathbb R^+ \times \mathbb R \; | \; w(t,x) \geq  \ell +r \big\} \Big)
    \end{aligned}
\end{equation}
where $\wt c, \wt C>0$ are some positive constants and we recall the definition of $g^{r, \ell}$ at \eqref{eq:gdef}.
Then we finally deduce
\begin{equation}
\begin{aligned}
     \mc Q(T) -\mc Q(T_0) & \leq -C\, \int_{T_0}^T \int_\mathbb R (g^{r, \ell}(t,y))^4 \dif y \dif t  \\
     &+
          \mc O(1) \cdot \msc L^2\Big( \big\{(t,x) \in \mathbb R^+ \times \mathbb R \; | \; w(t,x) \geq  \ell +r \big\} \Big)
          + C \cdot \Vert{\bs u_0}-\widehat{\bs u}\Vert_{\mathbf L^1}.
          \end{aligned}
\end{equation}
Rearranging, using \eqref{eq:Qbound}, and letting $T \to +\infty$, $T_0 \to 0^+$, proves the result.
\end{proof}

We are now ready to prove Theorem \ref{thm:decay}. We present the result for the first Riemann invariant $w$ the proof for the second being entirely analogous. By combining the two cases, we obtain the full statement of Theorem \ref{thm:decay}.
\begin{thm}\label{thm:kinit}
Let $\bs u: \mathbb R^+ \times \mathbb R \to \mc U$ be a bounded vanishing viscosity solution to \eqref{eq:systemi} with an initial datum $\bs u_0$ such that $\bs u_0-\widehat{\bs u} \in \mathbf L^1$ for some constant $\widehat{\bs u} \in \mc U$, with $\widehat w = w(\widehat{\bs u})$. Assume that the eigenvalue $\lambda_1$ is genuinely nonlinear.
Then, it holds
\begin{equation}
    \int_0^{+\infty} \int_{\mathbb R} \big(w(t,x)-\widehat w\big)^4 \dif x \dif t \leq \mc O(1) \cdot \Big(\Vert{\bs u_0}-\widehat{\bs u}\Vert_{\mathbf L^1} + \big(\Vert{\bs u_0}-\widehat{\bs u}\Vert_{\mathbf L^1}\big)^2\Big).
\end{equation}
\end{thm}
\begin{proof}
    We first prove that 
    \begin{equation}
    \int_0^{+\infty} \int_{\mathbb R} \big([w(t,x)- \widehat w]^+\big)^4 \dif x \dif t \leq \mc O(1) \cdot \Big(\Vert{\bs u_0}-\widehat{\bs u}\Vert_{\mathbf L^1} + \big(\Vert{\bs u_0}-\widehat{\bs u}\Vert_{\mathbf L^1}\big)^2\Big).
\end{equation}
The estimate for the negative part is entirely symmetric and is accordingly omitted.
Define $k \in \mathbb N, r>0$ by 
$$
k \doteq \left\lceil \frac{\bar  w-\widehat w}{\bar  {r}}\right\rceil, \qquad r \doteq  \frac{\bar  w-\widehat w}{k}
$$
where for $a \in \mathbb R$, the ceiling function $\lceil a\rceil$ denotes the smallest integer greater then $a$.
Then $r \cdot k =\bar  w-\widehat w$ and $0 < r \leq \bar { r}$. Define points $\ell_0, \ell_1, \ldots, \ell_{2(k-1)}$ by 
$$
\ell_0 \doteq \widehat w, \qquad \ell_i \doteq \widehat w + i \cdot \frac{r}{2}, \qquad i =1, \ldots,  2(k-1).
$$
Proposition \ref{prop:locdecayGNL} yields, by setting $\ell  = \ell_i$, 
\begin{equation}
    \begin{aligned}
    \int_0^{+\infty}\int_\mathbb R\big( g^{r,\ell_i}(t,x)\big)^4 \dif x \dif t & \leq\mc O(1) \cdot \Big(\Vert{\bs u_0}-\widehat{\bs u}\Vert_{\mathbf L^1} + \big(\Vert{\bs u_0}-\widehat{\bs u}\Vert_{\mathbf L^1}\big)^2\Big)\\
    & +\mc O(1) \cdot \msc L^2\Big( \big\{(t,x) \in \mathbb R^+ \times \mathbb R \; | \; w(t,x) \geq  \ell_i +r \big\} \Big).
        \end{aligned}
    \end{equation}
    Notice that 
    \begin{equation}\label{eq:cheb1}
    w(t,x) \geq \ell_i + r \quad \Longrightarrow \quad g^{r, \ell_{i+1}}(t,x) \geq \frac{r}{2}.
    \end{equation}
Now, for every $i = 0, \ldots, 2(k-1)-1$, using \eqref{eq:cheb1} and Markov's inequality, we estimate
\begin{equation}
\begin{aligned}
    \msc L^2\Big( \big\{(t,x) \in \mathbb R^+ \times \mathbb R \; | \; w(t,x) \geq  \ell_i +r \big\} \Big) & \leq \msc L^2\Big( \big\{(t,x) \in \mathbb R^+ \times \mathbb R \; | \; g^{r, \ell} \geq  r/2 \big\} \Big)\\
    & \leq  \frac{16}{r^4} \cdot \int_0^{+\infty}\int_\mathbb R\big( g^{r,\ell_{i+1}}(t,x)\big)^4 \dif x \dif t
    \end{aligned}
\end{equation}
Then, starting from $i = 0$, and applying the estimates above repeatedly, we obtain 
\begin{equation}
    \begin{aligned}
        \int_0^{+\infty}\int_\mathbb R\big( g^{r,\ell_{0}}(t,x)\big)^4 \dif x \dif t & \leq 2(k-1)\cdot \mc O(1) \cdot \Big(\Vert{\bs u_0}-\widehat{\bs u}\Vert_{\mathbf L^1} + \big(\Vert{\bs u_0}-\widehat{\bs u}\Vert_{\mathbf L^1}\big)^2\Big)\\
        & + \mc O(1)\cdot \msc L^2\Big( \big\{(t,x) \in \mathbb R^+ \times \mathbb R \; | \; w(t,x) \geq  \ell_{2(k-1)} +r \big\} \Big).
    \end{aligned}
\end{equation}
But now notice that $\ell_{2(k-1)}+r = \bar  w$, and by assumption $w(t,x) \leq \bar  w$ for a.e. $(t,x)$, therefore 
$$
 \msc L^2\Big( \big\{(t,x) \in \mathbb R^+ \times \mathbb R \; | \; w(t,x) \geq  \ell_{2(k-1)} +r \big\} \Big) = 0.
$$
Moreover, we clearly have
$$
    \int_0^{+\infty}\int_\mathbb R\big( [w(t,x)-\widehat w ]^+\big)^4 \dif x \dif t  \leq \mc O(1) \cdot  \frac{1}{r^4}\int_0^{+\infty}\int_\mathbb R\big( g^{r,\ell_{0}}(t,x)\big)^4 \dif x \dif t.
$$
Therefore we conclude that 
$$
 \int_0^{+\infty}\int_\mathbb R\big( [w(t,x)-\widehat w ]^+\big)^4 \dif x \dif t \leq \mc O(1) \cdot \Big(\Vert{\bs u_0}-\widehat{\bs u}\Vert_{\mathbf L^1} + \big(\Vert{\bs u_0}-\widehat{\bs u}\Vert_{\mathbf L^1}\big)^2\Big)
 $$
and this proves the result.
\end{proof}

\noindent {\bf Acknowledgements.}  The author wishes to thank Fabio Ancona for a careful reading of a preliminary version of the manuscript, and is partially supported by the GNAMPA - Indam Project 2025 \say{Rappresentazione lagrangiana per sistemi di leggi di conservazione ed equazioni cinetiche.   }

\end{document}